\documentclass[11pt]{amsart}
\usepackage{amsmath,amsthm}
\usepackage{bbm}
\usepackage{enumitem}
\usepackage{mathtools}
\usepackage{moreenum}
\usepackage{amssymb,verbatim}
\usepackage{upgreek}
\usepackage{xspace,cmap}
\usepackage{csquotes}
\usepackage{stmaryrd} 
\usepackage[colorlinks,citecolor=blue,urlcolor=black,linkcolor=blue]{hyperref}

\usepackage{mathtools}

\usepackage{pst-node}
\usepackage{tikz-cd} 
\usepackage{todonotes}

\renewcommand{\restriction}{\mathbin\upharpoonright}    
\newtheorem*{theorem*}{Theorem}
\newtheorem*{maintheorem*}{Main Theorem}
\newtheorem*{corollary*}{Corollary}
\newtheorem*{definition*}{Definition}

\newtheorem{theorem}{Theorem}[section]

\newtheorem{prop}[theorem]{Proposition}

\newtheorem*{claim*}{Claim}

\newtheorem{lemma}[theorem]{Lemma}
\newtheorem{cor}[theorem]{Corollary}

\newtheorem{question}{Question}
\newtheorem{fact}[theorem]{Fact}

\theoremstyle{definition}

\newtheorem{definition}[theorem]{Definition}
\newtheorem{notation}[theorem]{Notation}

\theoremstyle{remark}
\newtheorem{remark}[theorem]{Remark}

\newcommand\ZFC{\textnormal{ZFC}}
\newcommand\ZF{\textnormal{ZF}}
\newcommand\HOD{\textnormal{HOD}}
\newcommand\gHOD{\textnormal{gHOD}}
\newcommand\OD{\textnormal{OD}}
\newcommand{\Add}[2]{{\rm{Add}}({#1},{#2})}

\newcommand\cat[1]{{}^\curvearrowright #1}

\newcommand*\axiomfont[1]{\textsf{\textup{#1}}}

\newcommand\dc{\axiomfont{DC}}

\newcommand\gch{\axiomfont{GCH}}

\newcommand\ale[1]{\marginpar{Alejandro: #1}}

\newcommand{\seq}[2]{\langle{#1}~\vert~{#2}\rangle}
\newcommand{\map}[3]{{#1}:{#2}\longrightarrow{#3}}
\newcommand{\pmap}[4]{{#1}:{#2}\xrightarrow{#4}{#3}}
\newcommand{\ran}[1]{{{\rm{ran}}(#1)}}
\newcommand{\cof}[1]{{{\rm{cof}}(#1)}}
\newcommand{\Set}[2]{\{{#1}~\vert~{#2}\}}

\newcommand{\Lim}{{\mathrm{Lim}}}
\newcommand{\CCC}{{\mathbb{C}}}
\newcommand{\PPP}{{\mathbb{P}}}
\newcommand{\QQQ}{{\mathbb{Q}}}
\newcommand{\calC}{{\mathcal{C}}}
\newcommand{\calF}{{\mathcal{F}}}
\newcommand{\calU}{{\mathcal{U}}}

\DeclareMathOperator{\crit}{crit}

\DeclareMathOperator{\ob}{OB}
\DeclareMathOperator{\Ult}{Ult}
\DeclareMathOperator{\id}{id}

    \def\sq{\sqsubseteq}
    
    \newcommand{\one}{\mathop{1\hskip-3pt {\rm l}}}

\makeatletter
\newcommand{\tpitchfork}{%
  \vbox{
    \baselineskip\z@skip
    \lineskip-.52ex
    \lineskiplimit\maxdimen
    \m@th
    \ialign{##\crcr\hidewidth\smash{$-$}\hidewidth\crcr$\pitchfork$\crcr}
  }%
}
\makeatother

\def\s{\subseteq}
\def\forces{\Vdash}

\DeclareMathOperator{\otp}{otp}

\DeclareMathOperator{\ord}{Ord}

\renewcommand{\mid}{\mathrel{|}\allowbreak}
\newcommand{\mc}{\mathop{\mathrm{mc}}\nolimits}
\newcommand{\dom}{\mathop{\mathrm{dom}}\nolimits}

\newcommand{\Col}{\mathop{\mathrm{Col}}}

\title[]{The effects of forcing on exacting and ultraexacting cardinals}

\author[Lücke]{Philipp Lücke}
\address[Lücke]{Fachbereich Mathematik, Universit\"at Hamburg, Bundesstraße 55, Hamburg, 20146,
Germany}
\email{philipp.luecke@uni-hamburg.de}
\author[Poveda]{Alejandro Poveda}
\address[Poveda]{Departamento de Matemáticas, CUNEF Universidad,  Madrid, 28040, Spain}
\email{alejandro.poveda@cunef.edu}
\urladdr{https://www.alejandropovedaruzafa.com}


\begin{document}

\begin{abstract}
 Motivated by recent work of Aguilera--Bagaria--Lücke in \cite{ABL} and Aguilera--Bagaria--Goldberg--Lücke in \cite{ABLG}, we analyze various  effects of set-theoretic forcing upon the classes of exacting and ultraexacting cardinals. Using Magidor support products of Prikry forcings, we prove that  an I2-embedding yields a transitive, set-sized model of ZFC with a proper class of exacting cardinals. 
 We show that under Woodin's HOD Conjecture, no exacting cardinal can be a limit of exacting cardinals. In contrast, we prove that  models of ZF containing large cardinals beyond choice have forcing extensions that are  models of ZFC in which a regular  cardinal  is a stationary limit of ultraexacting cardinals. 
 We also show that, assuming the HOD Conjecture, an analogue of the classical Levy--Solovay theorem holds for exacting and ultraexacting cardinals. 
 Finally, we analyze the large cardinal properties possessed by exacting cardinals in HOD. We prove that if $\lambda$ is exacting and $V_\lambda$ satisfies the HOD Hypothesis, then $\lambda$ has strong large cardinal properties in $\HOD$. Carrying forward  forcing constructions from \cite{ABLG},  we start with a model of ZFC with an I2-embedding and produce a model of ZF in which the successor of an exacting cardinal $\lambda$ is  extendible in all models of the form $\HOD_x$ for  $x\s \lambda$.
\end{abstract}

\maketitle



\section{Introduction}\label{sec:Intro}

\emph{Large cardinal axioms} postulate the existence of uncountable cardinals whose properties resemble those of the smallest infinity, the cardinality $\aleph_0$ of the set of natural numbers. 
The combinatorial properties of these cardinals are strong enough to imply the consistency of $\ZFC$\footnote{This stands for the axioms of Zermelo-Fraenkel set theory with the Axiom of Choice.} and therefore  the Gödel's Incompleteness Theorems show that their existence can not be proven within this axiom system. 
Some of the earliest  examples of large cardinals are \emph{weakly inaccessible cardinals}, discovered by Hausdorff \cite{Hausdorff} in his work on the classification of  linear orders, and  \emph{weakly compact cardinals}, isolated by Erd\"os and Tarski \cite{ErdosTarski}    through their investigations of Ramsey-like theorems for uncountable cardinals. 
Besides their numerous applications across various fields of mathematics (see, for example, \cite{BagaCasa,AbadDodosTodorcevic,MagShe,Solovay}),  large cardinal axioms have played a pivotal role in  modern set theory in their capacity of canonical candidates for axioms extending the usual foundation of mathematics provided by the axioms of $\ZFC$. This thesis, vividly defended by  Gödel in the 30's (see \cite{Maddy}),  has been reinforced after a number of recent developments. Indeed, it has been demonstrated that large cardinal axioms enable a neat hierarchical classification of all mathematical theories extending $\ZFC$  through so-called \emph{equiconsistency results} (see, for example, \cite{KoellnerPlato}).

\smallskip

In his work on the construction of \emph{canonical inner models} for very large cardinals \cite{SEM1}, Woodin proved that strong large cardinal assumptions (\emph{extendible cardinals}) yield a surprising dichotomy: the set-theoretic universe ($V$) is either very close to Gödel's class of  \emph{Hereditarily Ordinal Definable sets} ($\HOD$) or it is very far from this inner model.  
%
Subsequent work of Goldberg in \cite{MR4693981} yields a  dichotomy akin to Woodin's assuming the existence of weaker large cardinals (i.e., \emph{strongly compact cardinals}).

The HOD dichotomy (see \S\ref{sec:HODhypothesis}) raises various far-reaching questions -- the most pressing one being if it is  a dichotomy at all. Indeed, it is widely open whether a model of $\ZFC$ that contains sufficiently strong large  cardinals (e.g., strongly compact cardinals) and is far from its $\HOD$ can be constructed starting from assumptions that are not outright incompatible with ZFC. 
Motivated by this, Woodin formulated the \emph{HOD Conjecture}, stating\footnote{The official formulation of the $\HOD$ Conjecture can be found in \S\ref{sec:HODhypothesis} below.} that $\ZFC$ proves that the existence of an extendible cardinal implies that $V$ and $\HOD$ are closed to each other. Woodin showed that this conjecture has  deep connections to central open problems in contemporary set theory; among others,  the construction of canonical inner models for supercompact cardinals and the consistency of \emph{large cardinals beyond choice} \cite{BagKoelWoo}. 

\smallskip

In \cite[Section 7]{SEM1}, Woodin discusses extensions of $\ZFC$ whose consistency would refute the $\HOD$ Conjecture. 
One of the axioms considered is the existence of a non-trivial elementary embedding $\map{j}{V_{\lambda+1}}{V_{\lambda+1}}$ with the property that for every $\Sigma_2$-formula $\varphi(v)$ and every $x\in V_{\lambda+1}$, the statement $\varphi(x)$ holds in $V$ if and only if the statement $\varphi(j(x))$ holds in $V$. 
The results of \cite{SEM1} then show that this assumption implies that $\lambda^+$ is a measurable cardinal in $\HOD$ and its consistency with the axioms of $\ZFC$ follows from the consistency of $\ZFC$  with \emph{Axiom I0}.\footnote{I.e., the existence of a non-trivial elementary embedding $\map{j}{L(V_{\lambda+1})}{L(V_{\lambda+1})}$ with $\crit(j)<\lambda$.} 
In particular, the consistency of $\ZFC$ with the existence of such an embedding at a cardinal $\lambda$ that is greater than an extendible cardinal would disprove the $\HOD$ Conjecture. 
Results of Woodin in \cite{SEM1} show that the consistency
of this constellation is derived from the consistency of $\ZF$  with certain large cardinals beyond choice. 

\smallskip

We now outline a line of research that arrives at the above axiom from a very different angle (see also \cite{PNAS}). 
In \cite{Huge}, Bagaria and the first author investigate strong model-theoretic reflection principles with the aim of finding a natural principle whose validity is equivalent to large cardinal axioms stronger than the category-theoretic postulate \emph{Vop\v{e}nka's Principle}.

This work left open the question of the consistency of the strongest of these principles. A continuation of this project with Aguilera in \cite{ABL} settled this question through the isolation of the notion of \emph{exacting cardinals}. 
This property can be formulated as a direct weakening of \emph{rank-Berkeleyness}, introduced by Goldberg and Schlutzenberg in \cite{GoSch24}. Here, a cardinal $\lambda$ is defined to be \emph{rank-Berkeley} if for all ordinals $\alpha<\lambda<\zeta$, there is a non-trivial elementary embedding $\map{j}{V_\zeta}{V_\zeta}$ such that  $\alpha<\crit(j)<\lambda$ and $\lambda$ is the first non-trivial fixed point of $j$. 
While the \emph{Kunen Inconsistency} implies that $\ZFC$ proves the non-existence of rank-Berkeley cardinals, the results of \cite{ABL}  show that a natural weakening of this property that still possesses many of the interesting combinatorial properties of rank-Berkeleyness is compatible with $\ZFC$:

\begin{definition}[{\cite{ABL}}]
    A cardinal $\lambda$ is called \emph{exacting} if for all ordinals $\alpha<\lambda<\zeta$, there is an elementary submodel $X$ of $V_\zeta$ with $V_\lambda\cup\{\lambda\}\s X$ and an elementary embedding $\map{j}{X}{V_\zeta}$ with $j\restriction\alpha=\id_\alpha$, $j\restriction \lambda\neq \id_{\lambda}$ and $j(\lambda)=\lambda$.  
\end{definition}

It can then be shown that the strongest reflection principles studied in \cite{Huge} hold at exacting cardinals. Moreover, the results of \cite{ABL} show that Axiom I0 implies the existence of a set-sized model of ZFC with an exacting cardinal. This result was later improved significantly in \cite{ABLG} (see also Corollary \ref{corollary:ExactingModels}\eqref{item:ExactingModels1} below), where it is shown that the existence of such a model follows from the existence of an \emph{I2-embedding}, {i.e.,} an elementary embedding $\map{j}{V}{M}$ of  $V$ into an inner model $M$ satisfying $V_\lambda\subseteq M$ holds, where $\lambda$ is the first non-trivial fixed point of $j$ (see \S\ref{subsection:RankIntoRank} for details).

It is shown in \cite{ABL} that the domains of elementary embeddings witnessing the exactingness of a cardinal $\lambda$ are always missing elements of $V_{\lambda+1}$. Therefore, a canonical way to strengthen the notion of exactingness is to demand that these domains contain certain elements of $V_{\lambda+1}$. 
The consistency proofs of \cite{ABL,ABLG}  first produce the elementary submodels that are to form the domains of the desired embeddings, and then obtain these functions as branches through trees of partial embeddings. In particular, these constructions cannot ensure that longer initial segments of these embeddings are contained in the domain models. These considerations  lead to the following definition: 
\begin{definition}[{\cite{ABL}}]
    A cardinal $\lambda$ is called \emph{ultraexacting} if for all ordinals $\alpha<\lambda<\zeta$, there is an elementary submodel $X$ of $V_\zeta$ with $V_\lambda\s X$ and an elementary embedding $\map{j}{X}{V_\zeta}$ with $j\restriction\alpha=\id_\alpha$, $j\restriction \lambda\neq \id_{\lambda}$, $j(\lambda)=\lambda$ and $j\restriction V_\lambda\in X$.
\end{definition}
In \cite{ABL}, this property is shown to be equivalent to what is, in a certain respect, a maximal strengthening of the model-theoretic reflection principles studied in \cite{Huge}. 
Moreover, in \cite{ABLG}, it is proven to be equivalent to the above-mentioned  property formulated by Woodin in \cite{SEM1}. 
Finally, it is shown in \cite{ABLG} how results in \cite{ABL} and \cite{Brink} can be combined to conclude that the consistency of $\ZFC$ with the existence of an ultraexacting cardinal is equivalent to the consistency of $\ZFC$ with Axiom I0. 

Besides answering the questions raised by the results of \cite{Huge}, the analysis of \cite{ABL} and \cite{ABLG} reveals that
exacting and ultraexacting cardinals possess various interesting structural properties. 
In particular, a short argument shows that exacting cardinals are singular cardinals that are regular in $\HOD$ and it therefore follows that the $\HOD$ Conjecture predicts the non-existence of such cardinals  above an extendible cardinal. 
Conversely, the consistency results of Woodin in \cite{SEM1} discussed above show that the consistency of $\ZFC$ with this constellation would follow from the consistency of $\ZF$ with certain large cardinals beyond choice. 
%
%

\smallskip

The goal of this paper is to further advance the systematic study of exacting and ultraexacting cardinals. 
These investigations are guided by far-reaching questions that connect these notions to central themes of set theory and whose answers should be based on substantial new insights into these topics.  
The first of these questions arises directly from the interplay between exacting cardinals and the covering properties of $\HOD$  described above:

\begin{question}\label{Q01}
    Does $\ZFC$ prove the non-existence of exacting cardinals above an extendible cardinal? 
\end{question}

We approach this question by considering the consequences of the given cardinal configuration and investigating whether the  consistency of these consequences can  be derived from large cardinal assumptions that are widely believed to be consistent. 
 The first  consequence we consider is the unprovability of bounds on the number of exacting cardinals. 
Note that, since the existence of an exacting cardinal above an ordinal $\alpha$ can be expressed by a $\Sigma_3$-formula with parameter $\alpha$ (see {\cite[Section 3]{ABL}}) and  extendible cardinals are $\Sigma_3$-correct (see {\cite[Proposition 23.10]{MR1994835}}), it follows that the existence of an exacting cardinal above an extendible cardinal implies the existence of a model of $\ZFC$ with a proper class of exacting cardinals. 
In contrast, the 
results of \cite{ABL} and \cite{ABLG} only produce models 
with a single exacting cardinal.\footnote{Notably, it can be arranged that the construction in the proof of {\cite[Theorem 5.4]{ABLG}} (see Corollary \ref{corollary:ExactingModels}\eqref{item:ExactingModels1} below) produces a model of the form $\HOD_{\vec{\kappa}}$, where $\vec{\kappa}$ is a strictly increasing $\omega$-sequence of cardinals that is Prikry-generic over $\HOD$ and whose supremum is an exacting cardinal. The analysis of \cite{ABL} then shows that this supremum is the only exacting cardinal in this model.} 
The following result, proven in \S\ref{section:ProperClassExacting}, now answers the arising question about the consistency strength of this consequence:

\begin{theorem}\label{thm: a proper class of exactings}
   If there exists an I2-embedding, then there exists a transitive set-sized model of the theory $$\ZFC ~ + ~ "\textit{There is a proper class of exacting cardinals}".$$ 
\end{theorem}

By the results of \cite{ABL}, an exacting cardinal must have countable cofinality; however, this does not preclude the possibility that an exacting cardinal might be an $\omega$-limit of exacting cardinals. In the model constructed in the proof of Theorem \ref{thm: a proper class of exactings}, none of the exacting cardinals obtained were limits of exacting cardinals, so it is natural to ask whether this is merely an artifact of the argument or a genuine limitation:

\begin{question}\label{Q02}
    Does $\ZFC$ prove that a limit of exacting cardinals is not an exacting cardinal?
\end{question}

The next observation, proven in \S\ref{sec:ExactingLimits}, shows that, at least under the \emph{Weak HOD Conjecture}\footnote{The statement of this weakening of the $\HOD$ Conjecture can be found in \S\ref{sec:HODhypothesis} below.}, the above question has an affirmative answer, thus establishing the optimality of the method employed in the proof of Theorem \ref{thm: a proper class of exactings}. 
In addition, in order to motivate further work, it gives an example how the validity of this conjecture would also rule out the existence of stronger limits of exacting cardinals.

\begin{theorem}\label{theorem:ExactingLimit}
   If the Weak HOD Conjecture is true, then $\ZFC$ proves the following statements: 
   \begin{enumerate}
       \item\label{item:ExactingLimit1} A limit of
   exacting cardinals is not an exacting cardinal. 

       \item\label{item:ExactingLimit2} For every cardinal $\kappa$ of uncountable cofinality, there is a closed unbounded subset of $\kappa$ that does not contain an exacting cardinal.
   \end{enumerate}    
\end{theorem}

These results raise the question of whether the given statements  can be proven without relying on some version of the $\HOD$ Conjecture. 
Our results will show that, in canonical scenarios in which the $\HOD$ Conjecture is false, these statements cannot be derived from the axioms of $\ZFC$.  
Below, we state a simplified version of the result deriving the strongest conclusion from the strongest assumption.  The definitions of the used assumptions can be found in \S\ref{sec:LCnoAC} and \S\ref{section:LCbeyondChoice}. The proof of this theorem can be found in \S\ref{sec:ExactingLimits}.

\begin{theorem}\label{thm:SimplifiedConsStatLimitUltra}
        If $\ZF$ is consistent with the existence of a cardinal that is both supercompact and totally Reinhardt, then $\ZFC$ is consistent with the existence of a regular cardinal that is a stationary limit of ultraexacting cardinals. 
\end{theorem}

In their classical paper \cite{LevySolovay}, Levy and Solovay demonstrated that small forcings do not create new measurable cardinals. 
More recently, Hamkins \cite{GapForcing} extended this result showing that the \emph{Levy--Solovay phenomenon} applies to  various other large cardinals, such as \emph{strong, Woodin, or supercompact}. 
In contrast, it turns out that the Levy--Solovay phenomenon for \emph{rank-into-rank embeddings} (see \S\ref{subsection:RankIntoRank}) is ostensibly more subtle, as demonstrated by Laver \cite{La97} and Woodin \cite{SEM1}. 
While Laver proved in \cite{La97} that the existence of an I3-embedding $\map{j}{V_\lambda}{V_\lambda}$ in a generic extension by a poset $\mathbb{P}\in H(\lambda)$ yields an I3-embedding $\map{i}{V_{\eta}}{V_{\eta}}$ for some cardinal $\eta\leq\lambda$ in the ground model, 
he also showed that it is possible that  forcing with the Levy collapse $\Col(\omega,\omega_1)$ introduces an I3-embedding $\map{j}{V_\lambda}{V_\lambda}$ at a cardinal $\lambda$ that has uncountable cofinality in the ground model (see \cite[p.2]{La97}).  Since the Kunen Inconsistency rules out the existence of non-trivial embedding $\map{j}{V_\eta}{V_\eta}$ for ordinals $\eta$ of uncountable cofinality, this shows that small forcings can create new rank-into-rank embeddings. 
 Woodin has also observed that similar phenomena take place for I0-embeddings.

Since the results of \cite{ABLG} show that exacting and ultraexacting cardinals can be characterized through direct strengthenings of rank-into-rank axioms (see Theorems~\ref{fact:CharExacting} and \ref{fact:CharUltraexacting}), it is now natural to ask whether similiar results can be proven for these notions or  whether the Levy--Solovay phenomenon can be extended to them:

\begin{question}\label{Q03}
    Can small forcings create new exacting and ultraexacting cardinals?
\end{question}

We approach this question by proving an analog of the main result of \cite{La97} for exacting and ultraexacting cardinals: 

\begin{theorem}\label{theorem:LevySolovayBeyondHOD}
  Let $\lambda$ be a cardinal and let $\PPP\in H(\lambda)$ be a partial order. 
  \begin{enumerate}
      \item\label{item:LevySolovayBeyondHOD1} If $\one\Vdash_\PPP``\textit{$\lambda$ is an exacting cardinal }"$, then the following statements hold: 
       \begin{enumerate}
           \item If $\lambda$ has countable cofinality, then $\lambda$ is an exacting cardinal. 

           \item If $\lambda$ has uncountable cofinality, then every closed unbounded subset of $\lambda$ contains an exacting cardinal. 
       \end{enumerate}

        \item\label{item:LevySolovayBeyondHOD2} If $\one\Vdash_\PPP``\textit{$\lambda$ is an ultraexacting cardinal }"$, then the following statements hold: 
       \begin{enumerate}
           \item If $\lambda$ has countable cofinality, then $\lambda$ is an ultraexacting cardinal. 

           \item If $\lambda$ has uncountable cofinality, then every closed unbounded subset of $\lambda$ contains an ultraexacting cardinal. 
       \end{enumerate}
  \end{enumerate}
\end{theorem}

In particular, if $\lambda$ is a singular cardinal with countable cofinality, it is equivalent to $\lambda$ to be exacting (respectively,  ultraexacting) and exacting (respectively, ultraexacting) in any generic extension by a poset in $H(\lambda)$.  Since the Weak HOD Conjecture establishes the non-existence of exacting cardinals that are  limit of exacting  cardinals (by Theorem~\ref{theorem:ExactingLimit}), the above theorem shows that the Weak HOD Conjecture yields the Levy--Solovay theorem for exacting and ultraexacting cardinals:

\begin{theorem}\label{theorem: WeakHODConjecture and LevySolovay}
    If the Weak HOD Conjecture is true, then $\ZFC$ proves the following statements: 
  \begin{enumerate}
   \item The following statements are equivalent for every cardinal $\lambda$: 
     \begin{enumerate}
         \item $\lambda$ is an exacting cardinal. 

        \item ${\one}\Vdash_\PPP\textit{$``\lambda$ is an exacting cardinal }"$ holds for every partial order $\PPP\in H(\lambda)$.   
        \item ${\one}\Vdash_\PPP\textit{$``\lambda$ is an exacting cardinal }"$ holds for some partial order $\PPP\in H(\lambda)$.   
     \end{enumerate}

    \item The following statements are equivalent for every cardinal $\lambda$: 
     \begin{enumerate}
         \item $\lambda$ is an ultraexacting cardinal. 

        \item ${\one}\Vdash_\PPP\textit{$``\lambda$ is an ultraexacting cardinal }"$ holds for every partial order $\PPP\in H(\lambda)$.   
        \item ${\one}\Vdash_\PPP\textit{$``\lambda$ is an ultraexacting cardinal }"$ holds for some partial order $\PPP\in H(\lambda)$.   
     \end{enumerate}
 \end{enumerate}
\end{theorem}
Whether or not the Weak HOD Conjecture is necessary to prove the above theorem is an open problem (see Question~\ref{que:isWHODHnecessary}).


%



\smallskip

In the final part of the manuscript, we investigate the extent of large cardinal properties in $\HOD$ that  exacting cardinals and their successors provably possess, as well as the degree to which these properties can consistently be strengthened. 
This general line of research  traces back to work of Woodin \cite{SEM1}  and has since been continued by  Cheng--Hamkins--Friedman \cite{ChengFriedmanHamkins}, Apter--Friedman--Fuchs \cite{ApterFriedmanFuchs} and Goldberg--Osinski--Poveda \cite{GOP}. The basic idea is: if the $\HOD$ Hypothesis holds, then $V$ is close to $\HOD$ and therefore the large cardinal hierarchies of $V$ and $\HOD$ are very similar –– at least, above the first extendible cardinal. This is precisely Woodin's \emph{Universality Theorem} (see {\cite[Theorem 3.26]{woodin_Midrasha}}) which, in its simpler formulation, asserts that, assuming the $\HOD$ conjecture,  if $\delta$ is the first extendible cardinal, then any cardinal $\kappa>\delta$  exhibiting \emph{virtually}  any large cardinal property retains this property in $\HOD$. In \cite{GOP}, it was shown that this result cannot be generalized to $\delta$ itself.

\smallskip

Since the $\HOD$ Hypothesis implies that there are no exacting cardinals  above the first extendible cardinal, it is necessary to study the large cardinal properties that exacting cardinals possess in $\HOD$ under weaker assumptions. 
 Moreover, the fact that exacting cardinals are singular in $V$ and regular in $\HOD$ directly rules out the direct transfer of various large cardinal properties between $V$ and $\HOD$. Our work takes the following question as a starting point:

\begin{question}\label{Q04}
  If $\lambda$ is an exacting cardinal, which of the following  statements holds in $\HOD$? 
 \begin{enumerate}
     \item $\lambda$ is a weakly compact cardinal. 

     \item $\lambda$ is a Woodin cardinal. 

     \item $\lambda$ is a limit of cardinals $\eta$ with the property that there exists an I3-embedding $\map{j}{V_\eta}{V_\eta}$. 
 \end{enumerate}
\end{question}

The first two statements are motivated by questions of Schlutzenberg regarding the large cardinal properties possesed in $\HOD$ by the first non-trivial fixed point of a Reinhardt embedding   (see {\cite[Question 3.8]{zbMATH07458790}}). The third statement is motivated by {\cite[Proposition 5.9]{ABLG}}, showing that exacting cardinals have this property in $V$.  
In \S\ref{sec: exactings in HOD}, we work towards answering this question by proving the following results that shows that the Weak $\HOD$ Conjecture provides an affirmative answer to the second and third statement in Question \ref{Q04}: 
%

\begin{theorem}\label{theorem:ExactingInHODIntro}
  If the Weak HOD Conjecture is true, then $\ZFC$ proves that the following statements hold for every exacting cardinal $\lambda$: 
  \begin{enumerate}
      \item\label{item:ExactingInHODIntro1} $\lambda$ is a Vop\v{e}nka cardinal in $\HOD$. 

      \item\label{item:ExactingInHODIntro2} In $\HOD$, for every closed unbounded subset $C$ of $\lambda$, there is  an I3-embedding whose critical sequence consists of elements of $C$.  
  \end{enumerate}
\end{theorem}

 The analysis underlying the proof of this result also shows that $\ZFC$ alone proves that exacting cardinals are stationary limits of measurables in $\HOD$ (see Corollary \ref{cor:ExactingStationaryLimitsHOD}).

 In \cite{ABLG}, it was shown that, over $\ZFC$, the existence of an  ultraexacting cardinal and Axiom I0 are equiconsistent. Moreover, the results of \cite{ABL} show that if $\lambda$ is ultraexacting, then $\lambda$ is exacting and $\lambda^+$ is measurable in $\HOD$. Since the existence of an exacting cardinal has weaker consistency strength than Axiom I2, it is now natural to ask if we can  get this configuration from assumptions weaker than I0 and if stronger large cardinal properties of $\lambda^+$  in $\HOD$ are possible. In \S\ref{sec: exacting from beyond choice}, we prove the following result for the $\ZF$-context:

\begin{theorem}\label{theorem: exacting + extendible} 
   If there exists an I2-embedding, then there exists a transitive  model of $\ZF$ with an exacting cardinal $\lambda$ such that $\lambda^+$ is extendible in $\HOD_x$ for every $x\subseteq\lambda$. 
\end{theorem}

The proof of this theorem (see Theorem~\ref{thm:generalizingABLG}) provides a generalization of  \cite[Corollary~5.6]{ABLG} which is relevant to Question~\ref{Q01}.

\smallskip

The structure of the paper is the following:  \S\ref{sec:preliminaries} collects a few relevant preliminaries; in  \S\ref{section:ProperClassExacting} we prove Theorem~\ref{thm: a proper class of exactings} and Theorem~\ref{theorem:ExactingLimit};   \S\ref{sec: exacting from beyond choice} is devoted to obtaining many exacting cardinals from cardinals beyond choice; \S\ref{sec: smallforcing}    discusses the Levy--Solovay phenomenon for exacting and ultraexacting cardinals proving Theorem~\ref{theorem:LevySolovayBeyondHOD} and  Theorem~\ref{theorem: WeakHODConjecture and LevySolovay}. In \S\ref{sec:ExactingInHOD} we prove Theorem~\ref{theorem:ExactingInHODIntro} and Theorem~\ref{theorem: exacting + extendible}. Finally, \S\ref{sec:openproblems} garners a few open problems.


\section{Preliminaries}\label{sec:preliminaries}


\subsection{HOD}\label{sec:HODhypothesis}
A set $X$ is \emph{ordinal definable}  if there is a set-theoretic formula $\varphi(v_0,\ldots,v_n)$ and ordinals $\alpha_0,\dots,\alpha_{n-1}$ such that $X=\Set{x}{\varphi(x,\alpha_0,\dots,\alpha_n)}$. We let $\OD$ denote the class of all ordinal definable sets. Additionally, a set $X$ is  \emph{hereditarily ordinal definable} if the transitive closure 
of $\{X\}$ is contained in $\OD$ and we let $\HOD$ denote the class of all hereditarily ordinal definable sets. 
In \cite{SEM1}, Woodin  defined a cardinal $\kappa$ to be \emph{$\omega$-strongly measurable in HOD} if for some $\eta$ with $(2^\eta)^\HOD < \kappa$, there is no partition  of the set $E^\kappa_\omega=\Set{\xi< \kappa}{\cof{\xi} = \omega}$ into $\eta$-many stationary sets that is an element of $\HOD$.  
Moreover, given a cardinal $\kappa$, he defined an inner model $M$ of $\ZFC$ to be a \emph{weak extender model for the supercompactness of $\kappa$} if for all \(X\in M\), there is a normal fine \({<}\kappa\)-complete ultrafilter
\(\mathcal U\) on \(\mathcal{P}_\kappa(X)\cap M\) with \(\mathcal U\cap M\in M\). Note that this definition implies that  \(\kappa\) is a supercompact cardinal in both \(M\) and \(V\).  
As it turns, weak extender models posses strong covering properties which make them close to the universe of sets. For instance, if $M$ is a weak extender model for the supercompactness of $\kappa$, then every singular cardinal $\lambda>\kappa$ is singular in $M$ and $\lambda^{+M}=\lambda$.

\smallskip

Woodin has proven the following striking dichotomy:

\begin{theorem}[The $\HOD$ Dichotomy, Woodin, \cite{SEM1}]\label{thm:HODdichotomy}
    If \(\delta\) is an extendible cardinal then exactly one of the following holds:
    \begin{enumerate}
      \item\label{item:HODclose} \(\HOD\) is a weak extender model for the supercompactness of \(\delta\).
        \item\label{item:HODfar} Every regular cardinal \(\kappa \geq \delta\) is \(\omega\)-strongly measurable in \(\HOD\).
    \end{enumerate}
\end{theorem}

Succinctly speaking, scenario \eqref{item:HODclose} states that $V$ is \emph{close} to $\HOD$, while scenario \eqref{item:HODfar} states that it is \emph{far} from $\HOD$. Which of these two antagonistic scenarios prevails is a major open question. The hypothesis and conjecture that settle this question are the following:

\begin{definition}[The HOD Hypothesis, Woodin, {\cite{SEM1}}]
  There is a proper class of regular cardinals that are not $\omega$-strongly measurable in $\HOD$.   
\end{definition}

\begin{definition}[Woodin, {\cite{SEM1}}]\hfill
    \begin{enumerate}
        \item The \emph{$\HOD$ Conjecture} is the assertion that the theory
        $$\ZFC+\textit{``There is an extendible cardinal"}$$
        proves the $\HOD$ Hypothesis. 
        
        \item The \emph{Weak $\HOD$ Conjecture} is the assertion that the theory
        $$\ZFC+\textit{``There is an extendible cardinal with a huge cardinal above"}$$
        proves the $\HOD$ Hypothesis.
    \end{enumerate}
\end{definition}

We end this section by considering strengthenings of the axiom $``V=\HOD"$ that we will use in our arguments. 
The axiom $``V=\gHOD$" asserts that the universe of sets is equal to the class $$\gHOD ~ = ~ \Set{x}{\forall \mathbb{P} ~  \one\forces_{\mathbb{P}}\check{x}\in\HOD}$$ (see \cite[\S4]{Geology}). This axiom  can be forced  over any given model of $\ZFC+\mathrm{GCH}$
 via a class forcing of any desired degree of directed closure  \cite{McAloon}. In fact, McAloon's method yields the consistency of a refined version of axiom $V=\mathrm{gHOD}$ called \emph{Continuum Coding Axiom} ($\mathrm{CCA}$) in work of Reitz \cite{Reitz}. The $\mathrm{CCA}$ is the axiom asserting that given any set of ordinals $a\s \alpha$ there is  another ordinal $\theta$ such that
 $$a ~ = ~ \Set{\beta < \alpha}{ 2^{\aleph_\theta+\beta+1}=\aleph_{\theta+\beta+2}}.$$


\subsection{Rank-into-rank embeddings}\label{subsection:RankIntoRank}
In \cite{KunenIncon}, Kunen proved the celebrated \emph{Kunen Inconsistency Theorem}, asserting that the Axiom of Choice precludes the existence of a nontrivial elementary embedding $$\map{j}{V_{\lambda+2}}{V_{\lambda+2}}.$$ A natural long-standing open question is whether Kunen's theorem is optimal, or whether the existence of elementary embeddings between smaller rank-initial segments of the universe of sets is already incompatible with the Axiom of Choice. This question motivated the study of so-called \emph{rank-into-rank embeddings}.

Given a limit ordinal $\lambda$, we refer to non-trivial elementary embeddings $\map{i}{V_\lambda}{V_\lambda}$ as \emph{I3-embeddings}. 
As usual, we define the \emph{critical sequence} of an I3-embedding $i$ is the unique sequence $\seq{\kappa_n}{n<\omega}$ satisfying  $\kappa_0=\crit(i)$ and $\kappa_{n+1}=i(\kappa_n)$ for all $n<\omega$. 
Since $\lambda$ is assumed to be a limit ordinal, the Kunen Inconsistency ensures that $\sup_{n<\omega}\kappa_n=\lambda$ holds for every I3-embedding $\map{i}{V_\lambda}{V_\lambda}$ with critical sequence $\seq{\kappa_n}{n<\omega}$.

If $\map{i}{V_\lambda}{V_\lambda}$ is an I3-embedding, then $$\map{i_+}{V_{\lambda+1}}{V_{\lambda+1}}; ~ A\mapsto\bigcup\Set{i(A\cap V_\gamma)}{\gamma<\lambda}$$ is the unique $\Sigma_0$-elementary map from $V_{\lambda+1}$ into itself that extends $i$ (see \cite[Lemma 3.3]{MR3902806}). 
It is now possible to strengthen the notion of I3-embeddings by demanding that these embeddings are elementary for higher-order formulas, where the extension $i_+$ is used to map second-order parameters. More specifically, we say that an I3-embedding $\map{i}{V_\lambda}{V_\lambda}$ is \emph{$\Sigma^1_n$-elementary} for some natural number $n$ if 
 \begin{equation*}
  \begin{split}
    V_\lambda & \models\Phi(x_0,\ldots,x_{k-1},A_0,\ldots,A_{\ell-1}) \\
      &  ~ \Longleftrightarrow ~ V_\lambda\models\Phi(i(x_0),\ldots,i(x_{k-1}),i_+(A_0),\ldots,i_+(A_{\ell-1}))
  \end{split}
 \end{equation*}
 holds for every $\Sigma^1_n$-formula\footnote{See \cite[p. 295]{MR1940513}.} $\Phi(v_0^0,\ldots,v^0_{k-1},v^1_0,\ldots,v^1_{\ell-1})$ with first-order variables $v_0^0,\ldots,v^0_{k-1}$ and second-order variables $v^1_0,\ldots,v^1_{\ell-1}$ and all parameters $x_0,\ldots,x_{k-1}\in V_\lambda$ and  $A_0,\ldots,A_{\ell-1}\in V_{\lambda+1}$. 
 We then define an I3-embedding $j$ to be an \emph{I1-embedding} if it is $\Sigma^1_n$-elementary for all natural numbers $n$. Note that this is equivalent to demanding that $j_+$ is an elementary embedding of $V_\lambda+1$ into itself. For this reason, we also refer to non-trivial elementary embeddings $\map{i}{V_{\lambda+1}}{V_{\lambda+1}}$ as I1-embeddings.

 Recall that a non-trivial elementary embedding $\map{j}{V}{M}$ is  an \emph{I2-embedding} if $V_\lambda \subseteq M$ holds, where $\lambda$ is the \emph{first non-trivial fixed point} of $j$, {i.e.,} the least ordinal $\gamma>\crit(j)$ satisfying $j(\gamma)=\gamma$. If we define the critical sequence $\seq{\kappa_n}{n<\omega}$ as above, then the Kunen Inconsistency again ensures that $\lambda=\sup_{n<\omega}\kappa_n$ holds.  
 %
%
 Classical results of Martin and Gaifman--Powell now show that the ability to extend an I3-embedding $\map{i}{V_\lambda}{V_\lambda}$ to an I2-embedding $\map{j}{V}{M}$ is equivalent to certain correctness properties of the embedding $i$:

 \begin{theorem}[{\cite[Theorem 1.3]{La97}}]\label{theorem:I2Char}
     The following statements are equivalent for every I3-embedding $\map{i}{V_\lambda}{V_\lambda}$: 
     \begin{enumerate}
        \item The embedding $i$ can be extended to an I2-embedding $\map{j}{V}{M}$. 

         \item If $R$ is a well-founded relation on $V_\lambda$, then $i_+(R)$ is a well-founded relation on $V_\lambda$. 
         
         \item The embedding $i$ is $\Sigma^1_1$-correct. 

         \item The embedding $i$ is $\Sigma^1_2$-correct.  
     \end{enumerate}
 \end{theorem}

 Motivated by these equivalences, we also refer to $\Sigma^1_1$-correct I3-embeddings as I2-embeddings.


 The lifting $i_+$ of I3-embeddings $i$ can also be used to iterate such embeddings. If $\map{i}{V_\lambda}{V_\lambda}$ is an I3-embedding with critical sequence $\seq{\kappa_n}{n<\omega}$ and $\map{k}{V_\lambda}{V_\lambda}$ is another I3-embedding, then $k_+(i)$ is  an I3-embedding with critical sequence $\seq{k(\kappa_n)}{n<\omega}$ (see {\cite[Lemma 3.5]{MR3902806}}). 
 Moreover, if $i$ and $k$ are  $\Sigma^1_n$-correct, then both $k\circ i$ and $k_+(i)$ are $\Sigma^1_n$-correct (see {\cite[Theorem 2.4]{La97}}). 
 Given an I3-embedding $\map{i}{V_\lambda}{V_\lambda}$, we now define the \emph{iteration} of $i$ to be the unique system $$\seq{\map{i_{\ell,m}}{V_\lambda}{V_\lambda}}{\ell\leq m<\omega}$$ of elementary embeddings with $i_{0,1}=i$, $i_{\ell,\ell}=\id_{V_\lambda}$, $i_{\ell+1,\ell+2}=i_+(i_{\ell,\ell+1})$ and $i_{\ell,m+1}=i_{m,m+1}\circ i_{\ell,m}$ for all $\ell\leq m<\omega$. 
 We then know that for each $\ell<\omega$, the map $i_{\ell,\ell+1}$ is an I3-embedding with critical sequence $\seq{\kappa_{n+\ell}}{n<\omega}$. Moreover, a short computation then shows that $$i_{\ell+1,\ell+2} ~ = ~ (i_{\ell,\ell+1})_+(i_{\ell,\ell+1})$$ holds for all $\ell<\omega$. Finally, the results mentioned above directly ensure that, if the embedding $i$ is $\Sigma^1_n$-correct for some $n<\omega$, then all embeddings $i_{\ell,m}$ are $\Sigma^1_n$-correct. In particular, Theorem  \ref{theorem:I2Char} allows us to conclude that the existence of an I2-embedding with critical sequence $\seq{\kappa_n}{n<\omega}$ implies that for all $\ell<\omega$, there exists an I2-embedding with critical sequence $\seq{\kappa_{n+\ell}}{n<\omega}$.

 If $\map{j}{V}{M}$ is an I2-embedding with critical point $\kappa$ and least non-trivial fixed point $\lambda$ and $E$ is the $(\kappa,\lambda)$-extender induced by $j$, then the corresponding ultrapower embedding $\map{j_E}{V}{\Ult(V,E)}$  is again an I2-embedding and it agrees with $j$ on $V_{\lambda+1}$ (see {\cite[Lemma 26.1]{MR1994835}}). 
 Standard arguments (see {\cite[Section 3]{MR3902806}}) now show that I2-embeddings given by extenders can be \emph{iterated}; more precisely, if $E$ is a $(\kappa,\lambda)$-extender that induces an  I2-embedding $\map{j}{V}{M}$ with critical point $\kappa$ and least non-trivial fixed point $\lambda$, then there exists a commuting system $$\langle\seq{M_\alpha}{\alpha\in\ord},\seq{\map{j_{\alpha,\beta}}{M_\alpha}{M_\beta}}{\alpha\leq\beta\in\ord}\rangle$$ of inner models $M_\alpha$ and elementary embeddings $j_{\alpha,\beta}$ such that the following statements hold:
 \begin{itemize}
     \item $M_0=V$, $M_1=M$ and $j_{0,1}=j$. 

     \item If $\alpha\in\ord$, then $M_{\alpha+1}=\Ult(M_\alpha,j_{0,\alpha}(E))$ and $j_{\alpha,\alpha+1}$ is the corresponding ultrapower embedding. 

     \item If $\lambda$ is a limit ordinal, then $$\langle M_\lambda,\seq{j_{\alpha,\lambda}}{\alpha<\lambda}\rangle$$ is the direct limit of $$\langle\seq{M_\alpha}{\alpha<\lambda},\seq{\map{j_{\alpha,\beta}}{M_\alpha}{M_\beta}}{\alpha\leq\beta<\lambda}\rangle.$$
 \end{itemize}
 Routine computations then show that, if $\seq{\kappa_n}{n<\omega}$ is the critical sequence of $j$ and $m\leq n<\omega$, then we have $j_{m,n}(\kappa_m)=\kappa_n$,  $\crit(j_{n,n+1})=\kappa_n$, $j_{m,n}(\lambda)=\lambda$ and $V_\lambda\subseteq M_n$. It then follows that $j_{0,\omega}(\kappa)=\lambda$,  $\crit(j_{\omega,\omega+1})=\lambda$ and $V_\lambda\subseteq M_\alpha$ for all $\alpha\in\ord$. Moreover, it can be shown that the system $\seq{j_{\ell,m}\restriction V_\lambda}{\ell\leq m<\omega}$ of I3-embeddings is the iteration of $j\restriction V_\lambda$ (see the proof of {\cite[Proposition 3.8]{MR3902806}}).

 The following observation plays a central role in the consistency proofs of exactingness from I2-embeddings:

\begin{prop}\label{proposition:IterationFixedPoints}
If $\map{j}{V}{M}$ is an I2-embedding with critical point $\kappa$ and least non-trivial fixed point $\lambda$ that is given by a $(\kappa,\lambda)$-extender $E$ and  $$\langle\seq{M_\alpha}{\alpha\in\ord},\seq{\map{j_{\alpha,\beta}}{M_\alpha}{M_\beta}}{\alpha\leq\beta\in\ord}\rangle$$ is the iteration of $V$ and $E$, then the following statements hold:  
 \begin{enumerate}
     \item\label{prop:InvEmb1} We  have $$x\in M_\omega ~ \Longleftrightarrow ~ j(x)\in M_\omega$$ for every set $x$ and $$\map{j\restriction M_\omega}{M_\omega}{M_\omega}$$ is an elementary embedding. 

     \item\label{prop:InvEmb2} We have $$j(j_{0,\omega}(x)) ~ = ~ j_{0,\omega}(x)$$ for every set $x$. 
 \end{enumerate}
\end{prop}

\begin{proof}
  First, note that, in $V$, the above system of inner models and elementary embeddings is definable from the paramater $E$. 
  In addition, the system  $$\langle\seq{M_{1+\alpha}}{\alpha\in\ord},\seq{\map{j_{1+\alpha,1+\beta}}{M_{1+\alpha}}{M_{1+\beta}}}{\alpha\leq\beta\in\ord}\rangle$$ is the iteration of $M_1$ and $j(E)$, and this system is definable in $M_1$ by the same formulas and the parameter $j(E)$. 
  Given $\alpha\leq\omega$, this shows that  $$x\in M_\alpha ~ \Longleftrightarrow ~ j(x)\in M_{1+\alpha}$$ holds for every set $x$ and $$\map{j\restriction M_\alpha}{M_\alpha}{M_{1+\alpha}}$$ is an elementary embedding. This directly implies \eqref{prop:InvEmb1}. In addition, given $\alpha\leq\beta\leq\omega$, we have   $$j(j_{\alpha,\beta}(x)) ~ = ~ j_{1+\alpha,1+\beta}(j(x))$$  for every $x\in M_\alpha$. Since $j_{1,\omega}\circ j=j_{0,\omega}$, this directly implies \eqref{prop:InvEmb2}. 
\end{proof}

 For later use, we close this section by proving that it is possible to start in a model of the $\gch$ with an I2-embedding and force $V=\gHOD$ to hold while preserving the existence of an I2-embedding:

\begin{lemma}\label{lemma:I2andVHOD}
    If the $\gch$ holds and $\map{j}{V}{M}$ is an I2-embedding with critical point $\kappa$ and least non-trivial fixed point $\lambda$ that is induced by a $(\kappa,\lambda)$-extender, then there is a cofinality-preserving generic extension $V[G]$ such that $j$ lifts to an I2-embedding in $V[G]$ and $``V=\gHOD$" holds in $V[G]_\lambda$.

    Moreover, $\mathrm{CCA}$ holds in $V[G]_\lambda$.
\end{lemma}

\begin{proof}
 Let $\mathcal{I}\s \lambda$ denote the set of all inaccessible cardinals less than $\lambda$. Let $$\PPP_\lambda  ~ = ~ \seq{\PPP_\alpha;\dot{\QQQ}_\alpha}{\alpha<\lambda}$$ denote the Easton support  forcing iteration that is non-trivial only at those   $\min(\mathcal{I})<\alpha\in\mathcal{I}$ such that $\sup(\mathcal{I}\cap \alpha)<\alpha$ holds, in which case $\PPP_\alpha$ forces $\dot{\QQQ}_\alpha$ to be the canonical ${<}\sup(\mathcal{I}\cap \kappa)^+$-directed-closed forcing coding $V[\dot{G}]_\alpha$ into the continuum pattern of $V[\dot{G}]_\alpha$ (see \cite{McAloon}). 
 Then $\PPP_\lambda$ is $\sigma$-closed and, since the $\gch$ holds in the ground model, standard arguments (see {\cite[Section 7]{CummingsHandbook}})  show that forcing with $\PPP_\lambda$ preserves all cofinalities. Moreover, the above definition ensures that $V[H]_\lambda$  satisfy $``V=\gHOD$" whenever $H$ is $\mathbb{P}_\lambda$-generic over $V$.
 Therefore, it suffices to check that $\map{j}{V}{M}$ lifts to an I2-embedding in some $\mathbb{P}_\lambda$-generic extension of $V$.

 Let $\seq{\kappa_n}{n<\omega}$ denote the critical sequence of $j$. Then $\kappa_0$ is the critical point of $j$ and $\lambda=\sup_{n<\omega}\kappa_n$. 
 Since  the partial order $\PPP_\alpha$ is uniformly definable in $V_\lambda$ from the parameter $\alpha<\lambda$, it follows that $j(\PPP_{\kappa_n})=\PPP_{\kappa_{n+1}}$ and $j(\PPP_{[\kappa_n,\kappa_{n+1})})=\PPP_{[\kappa_{n+1},\kappa_{n+2})}$ for all $n<\omega$. This shows that $j[\PPP_\lambda]$ is a subset of $\PPP_\lambda$. 
  Given $n<\omega$, it is also obvious that $|\mathbb{P}_\alpha|<\kappa_n$ holds for all $\alpha<\kappa_n$ and every condition in $\mathbb{P}_{\kappa_n}$ forces $\mathbb{P}_{[\kappa_n,\kappa_{n+1})}$ to be ${<}\kappa_n^+$-directed closed. 
  %

\begin{claim*}
  Let $n<\omega$, let $G$ be $\PPP_{\kappa_{n+1}}$-generic over $V$, let $\bar{G}$ be the filter on $\PPP_{\kappa_n}$ induced by $G$ and let $H$ be the filter on $\PPP_{[\kappa_n,\kappa_{n+1})}^{\bar{G}}$ induced by $G$. If $j[\bar{G}]\subseteq G$ holds and  $\map{j_*}{V[\bar{G}]}{M[G]}$ is the corresponding lifting of $j$, then $j_*[H]$ is a subset of $\PPP_{[\kappa_{n+1},\kappa_{n+2})}^G$ with a lower bound. 
  %
\end{claim*}

\begin{proof}[Proof of claim]
  First, note that, since $\kappa_{n+1}$ is an inaccessible limit of inaccessible cardinals and $\PPP_\lambda$ has Easton support, it follows that $\PPP_{\kappa_{n+1}}$ has cardinality at most $\kappa_{n+1}$ in $V$ and $\PPP_{[\kappa_n,\kappa_{n+1})}^{\bar{G}}$ has cardinality at most $\kappa_{n+1}$ in $V[G]$. In particular, we know that $H$ has cardinality at most $\kappa_{n+1}$ in $V[G]$. 
  Since $j(\PPP_{[\kappa_n,\kappa_{n+1})}) =  \PPP_{[\kappa_{n+1},\kappa_{n+2})}$, we also know that $$j_*[H] ~ \subseteq ~ j_*(\PPP_{[\kappa_n,\kappa_{n+1})}^{\bar{G}}) ~ = ~ \PPP_{[\kappa_{n+1},\kappa_{n+2})}^G.$$ Finally, since $j_*$ is an elementary embedding, the directedness of $H$ in $\PPP_{[\kappa_n,\kappa_{n+1})}^{\bar{G}}$ implies the directedness of $j_*[H]$ in $\PPP_{[\kappa_{n+1},\kappa_{n+2})}^G$. 
  The statement of the claim now follows from the earlier observation that $\PPP_{[\kappa_{n+1},\kappa_{n+2})}^G$ is ${<}\kappa_{n+1}^+$-directed closed in $V[G]$. 
  %
\end{proof}

 We  now use this claim to inductively construct a sequence $\seq{p_n}{n<\omega}$ consisting of conditions $p_n$ in $\PPP_{\kappa_{n+1}}$ satisfying the following statements: 
   \begin{enumerate}
       \item If $m<n<\omega$, then $p_n\restriction\kappa_{m+1}\leq_{\PPP_{\kappa_{m+1}}}p_m$. 

       \item If $n<\omega$, $G$ is $\PPP_{\kappa_{n+1}}$-generic over $V$ with $p_n\in G$, $\bar{G}$ is the filter on $\PPP_{\kappa_n}$ induced by $G$ and $H$ is the filter on $\PPP_{[\kappa_n,\kappa_{n+1})}^{\bar{G}}$ induced by $G$, then $j[\bar{G}]\subseteq G$ and $p_{n+1}\restriction[\kappa_{n+1},\kappa_{n+2})$ is a lower bound of $j_*[H]$ in $\PPP_{[\kappa_{n+1},\kappa_{n+2})}^G$, where $\map{j_*}{V[\bar{G}]}{M[G]}$ is the canonical lifting of $j$. 
   \end{enumerate}
 
 First, define $p_0$ to be the trivial condition in $\PPP_{\kappa_1}$. Since $j\restriction\PPP_{\kappa_0}=\id_{\PPP_{\kappa_0}}$ and $\PPP_{\kappa_0}$ is a direct limit, we obviously have $j[\bar{G}]\subseteq G$ whenever $G$ is $\PPP_{\kappa_1}$-generic over $V$ and $\bar{G}$ is the filter on $\PPP_{\kappa_0}$ induced by $G$. 
 Now, assume that $n<\omega$ and there exists a condition $p_n$ in $\PPP_{\kappa_{n+1}}$ with the desired properties. By the above claim, these properties ensure that there is a $\PPP_{\kappa_{n+1}}$-name $\dot{q}$ for a condition in $\PPP_{[\kappa_{n+1},\kappa_{n+2})}$ with the property that whenever $G$ is $\PPP_{\kappa_{n+1}}$-generic over $V$ with $p_n\in G$,  $\bar{G}$ is the filter on $\PPP_{\kappa_n}$ induced by $G$,  $H$ is the filter on $\PPP_{[\kappa_n,\kappa_{n+1})}^{\bar{G}}$ induced by $G$ and $\map{j_*}{V[\bar{G}]}{M[G]}$ is the canonical lifting of $j$ enabled by the fact that $j[\bar{G}]\subseteq G$ holds, then $\dot{q}^G$ is a lower bound of $j_*[H]$ in $\PPP_{[\kappa_{n+1},\kappa_{n+2})}^G$. 
 Define $p_{n+1}$ to be a condition in $\PPP_{\kappa_{n+2}}$ that gets mapped below $\langle p_n,\dot{q}\rangle$ by the canonical dense embedding of $\PPP_{\kappa_{n+2}}$ into $\PPP_{\kappa_{n+1}}*\PPP_{[\kappa_{n+1},\kappa_{n+2})}$. 
 Then $p_{n+1}\restriction\kappa_{n+1}\leq_{\PPP_{\kappa_{n+1}}}p_n$ and $p_{n+1}\restriction[\kappa_{n+1},\kappa_{n+2})$ is a lower bound of $j_*[H]$ in $\PPP_{[\kappa_{n+1},\kappa_{n+2})}^G$, whenever $G$ is $\PPP_{\kappa_{n+1}}$-generic over $V$ with $p_n\in G$, $\bar{G}$ is the filter on $\PPP_{\kappa_n}$ induced by $G$, $H$ is the filter on $\PPP_{[\kappa_n,\kappa_{n+1})}^{\bar{G}}$ induced by $G$ and $\map{j_*}{V[\bar{G}]}{M[G]}$ is the corresponding lifting of $j$. 
 Finally, let $G$ be $\PPP_{\kappa_{n+2}}$-generic over $V$ with $p_{n+1}\in G$ and let $\bar{G}$ be the filter on $\PPP_{\kappa_{n+1}}$ induced by $G$. Fix a condition $q$ in $\bar{G}$. Since $p_n$ is an element of $\bar{G}$, we know that $j(q)\restriction\kappa_{n+1}=j(q\restriction\kappa_n)\in\bar{G}$. Moreover, our construction also ensures that 
 \begin{equation*}
   \begin{split}
       p_{n+1}\restriction[\kappa_{n+1},\kappa_{n+2}) ~ & \leq_{\PPP_{[\kappa_{n+1},\kappa_{n+2})}^G} ~ \dot{q}^G \\ & \leq_{\PPP_{[\kappa_{n+1},\kappa_{n+2})}^G} ~ j(q\restriction[\kappa_n,\kappa_{n+1})) ~ = ~ j(q)\restriction[\kappa_{n+1},\kappa_{n+2}).
   \end{split}
 \end{equation*}
 This shows that there is a condition $r\in G$ with $r\leq_{\PPP_{\kappa_{n+2}}}j(q)$. In particular, we know that $j[\bar{G}]\subseteq G$ holds in this case.

 Since all of the partial orders in our iteration are countably closed, we can now find a condition $p$ in $\PPP_\lambda$ satisfying $p\restriction\kappa_{n+1}\leq_{\PPP_{\kappa_{n+1}}}p_n$ for all $n<\omega$. Let $F$ be $\PPP_\lambda$-generic over $V$ with $p\in F$ and assume, towards a contradiction, that there is a condition $q$ in $F$ with $j(q)\notin F$. Let $r$ be a common strengthening of $p$ and $q$ in $F$. Then $r\not\leq_{\PPP_\lambda}j(q)$ and we can find a minimal $n<\omega$ with the property that $$r\restriction\kappa_{n+1} ~ \not\leq_{\PPP_{\kappa_{n+1}}} ~ j(q)\restriction\kappa_{n+1} ~ = ~ j(q\restriction\kappa_n).$$ We then know $n>0$, because $r\restriction\kappa_0\leq_{\PPP_{\kappa_0}}q\restriction\kappa_0$ and Easton support ensures that $j(q)\restriction[\kappa_0,\kappa_1)$ is the trivial condition in $\PPP_{[\kappa_0,\kappa_1)}$. 
  The minimality of $n$ then ensures that $r\restriction\kappa_n \leq_{\PPP_{\kappa_n}} j(q)\restriction\kappa_n$ and there exists a filter $G$ on $\PPP_{\kappa_n}$ that is generic over $V$, contains $r\restriction\kappa_n$ and has the property that $r\restriction[\kappa_n,\kappa_{n+1})$ is not a strengthening of $j(q)\restriction[\kappa_n,\kappa_{n+1})$ in $\PPP_{[\kappa_n,\kappa_{n+1})}^G$. 
  Let $\bar{G}$ denote the filter on $\PPP_{\kappa_{n-1}}$ induced by $G$ and let $H$ denote the filter on $\PPP_{[\kappa_{n-1},\kappa_n)}^{\bar{G}}$ induced by $G$.   Since $p_{n-1}$ is an element of $G$, we have $j[\bar{G}]\subseteq G$  and  $p_n\restriction[\kappa_n,\kappa_{n+1})$ is a lower bound of $j_*[H]$ in $\PPP_{[\kappa_n,\kappa_{n+1})}^G$, where $\map{j_*}{V[\bar{G}]}{M[G]}$ is the canonical lifting of $j$. Moreover, since $q\restriction\kappa_n$ is an element of $G$ and therefore $q\restriction[\kappa_{n-1},\kappa_n)$ is an element of $H$, it follows that the condition $j(q\restriction[\kappa_{n-1},\kappa_n))=j(q)\restriction[\kappa_n,\kappa_{n+1})$ is an element of $j_*[H]$. In combination, this shows that $$r\restriction[\kappa_n,\kappa_{n+1}) ~ \leq_{\PPP_{[\kappa_n,\kappa_{n+1})}^G} ~ j(q)\restriction[\kappa_n,\kappa_{n+1}),$$ a contradiction. 
  These computations show that $j[F]$ is contained in $F$.

  Given $\ell<\omega$, we let $F_\ell$ denote the filter on $\PPP_{\kappa_\ell}$ induced by $F$. Define $\bar{F}$ to be the set of all conditions $p$ in $j(\PPP_\lambda)$ with the property that there exists $n<\omega$, $q\in F_{n+1}$ and $r\in F$ such that $$\langle q,j(r)\restriction[\kappa_{n+1},\lambda)\rangle ~ \leq_{\PPP_{\kappa_{n+1}*j(\PPP_{[\kappa_n,\lambda)})}} ~ \langle p\restriction\kappa_{n+1}, ~ p\restriction[\kappa_{n+1},\lambda)\rangle$$ holds in $M$. 
  The fact that $j[F]$ is a subset of $F$ then implies that $\bar{F}$ is a filter on $j(\PPP_\lambda)$ that contains $j[F]$.

  \begin{claim*}
    $\bar{F}$ is $j(\PPP_\lambda)$-generic over $M$. 
  \end{claim*}

  \begin{proof}[Proof of the claim]
    Let $D$ be a dense open subset of $j(\PPP_\lambda)$ in $M$. Since $M$ is an ultrapower of $V$ using a $(\kappa,\lambda)$-extender, we can find a function $f$ with domain $V_\lambda$ and $y\in V_\lambda$ with  $j(f)(y)=D$ (see e.g., \cite[\S3]{PovTwo}). We may assume that $f(x)$ is a dense open subset of $\PPP$ for every $x\in V_\lambda$. 
    Fix $n<\omega$ with $y\in V_{\kappa_{n+1}}$. 
    It is then easy to see that $$\Set{p\restriction[\kappa_n,\lambda)}{p\in f(x), ~ p\restriction\kappa_n\in F}$$ is a dense subset of $\PPP_{[\kappa_n,\lambda)}^{F_n}$ in $V[F_n]$ for every $x\in V_\lambda$. 
    Since $V_{\kappa_n}$ has cardinality $\kappa_n$ in the ground model and $\PPP_{[\kappa_n,\lambda)}^{F_n}$ is ${<}\kappa_n^+$-directed closed in $V[F_n]$, we can find a condition $r$ in $F$ with the property that for every $x\in V_{\kappa_n}$, the set of conditions $q$ in $\PPP_{\kappa_n}$ with the property that there exists $p\in f(x)$ with $$\langle q, ~  r\restriction[\kappa_n,\lambda)\rangle ~ \leq_{\PPP_{\kappa_n}*\PPP_{[\kappa_n,\lambda)}} ~ \langle p\restriction\kappa_n, ~ p\restriction[\kappa_n,\lambda)\rangle$$
    is dense below $r\restriction\kappa_n$ in $\PPP_{\kappa_n}$. 
    Elementarity now implies that, in $M$, the set of conditions $q$ in $\PPP_{\kappa_{n+1}}$ with the property that there exists $p\in D$ with 
    \begin{equation}\label{equation:DenseInUltrapower}
        \langle q, ~  j(r)\restriction[\kappa_{n+1},\lambda)\rangle ~ \leq_{\PPP_{\kappa_{n+1}}*j(\PPP_{[\kappa_n,\lambda)})} ~ \langle p\restriction\kappa_{n+1}, ~ p\restriction[\kappa_{n+1},\lambda)\rangle
    \end{equation}
    is dense below $j(r)\restriction\kappa_{n+1}$ in $\PPP_{\kappa_{n+1}}$. 
    Since the above computations show that $j(r)\restriction\kappa_{n+1}=j(r\restriction\kappa_n)\in F$, we can find  $q\in F_{n+1}$ and $p\in D$ such that \eqref{equation:DenseInUltrapower} holds. Then $p$ is an element of $D\cap\bar{F}$.  
  \end{proof}

  The above claim shows that we can lift $j$ to an elementary embedding  $$\map{j}{V[F]}{M[\bar{F}]}.$$

  \begin{claim*}
      $V[F]_\lambda\subseteq M[\bar{F}]$. 
  \end{claim*}

  \begin{proof}[Proof of the claim]
    Fix $n<\omega$ and $x\in V[F]_{\kappa_n}$. The definition of $\PPP_\lambda$ then ensures that $x$ is of the form $\tau^{F_n}$ for some $\PPP_{\kappa_n}$-name $\tau$ in $V_\lambda$. By the definition of $\bar{F}$, we now have  $F_n\in M[\bar{F}]$ and hence $x\in M[\bar{F}]$. 
  \end{proof}

  We now know that $j$ lifts to an I2-embedding in $V[F]$. 
\end{proof}


\subsection{Exacting and ultraexacting cardinals}\label{sec:Exacting}
In this short section, we recall some basic results about exacting and ultraexacting cardinals from \cite{ABL} and \cite{ABLG} that will be used throughout this paper. 
%
%
The first of these results show that the given properties can also be characterized by the existence of a single embedding. Following \cite{Ba:CC}, for every natural number $n$, we let $C^{(n)}$ denote the class of all $\Sigma_n$-correct cardinals, {i.e.,} the class of all ordinals $\alpha$ with the property that $V_\alpha$ is a $\Sigma_n$-elementary submodel of the set-theoretic universe $V$.

\begin{lemma}[{\cite[Lemma 2.3]{ABL}}]\label{lemma:CharExactingSingleEmb}
   A cardinal $\lambda$ is exacting if and only if there exists   $\lambda<\eta\in C^{(1)}$,   $\lambda<\zeta\in C^{(2)}$, an elementary submodel $X$ of $V_\eta$ with $V_\lambda\cup\{\lambda\}\subseteq X$ and an elementary embedding $\map{j}{X}{V_\zeta}$ satisfying $j\restriction\lambda\neq\id_\lambda$ and $j(\lambda)=\lambda$. 
\end{lemma}

\begin{lemma}[{\cite[Lemma 3.2]{ABL}}]
   A cardinal $\lambda$ is ultraexacting if and only if there exists $\lambda<\eta\in C^{(1)}$, $\lambda<\zeta\in C^{(2)}$, an elementary submodel $X$ of $V_\eta$ with $V_\lambda\subseteq X$ and an elementary embedding $\map{j}{X}{V_\zeta}$ satisfying $j\restriction\lambda\neq\id_\lambda$, $j(\lambda)=\lambda$ and $j\restriction V_\lambda\in X$. 
\end{lemma}

The following result shows that exactingness is a direct strengthening of Axiom I3:

\begin{theorem}[{\cite[Theorem 3.4]{ABLG}}]\label{fact:CharExacting}
    A cardinal $\lambda$ is exacting if and only if for every ordinal definable subset $A$ of $V_{\lambda+1}$, there are $x,y\in A$ and a non-trivial elementary embedding $\map{j}{\langle V_\lambda,\in,x\rangle}{\langle V_\lambda,\in,y\rangle}$ with $\crit(j)<\lambda$. 
\end{theorem}

An analogous result characterizes ultraexactingness as a strengthening of Axiom I1:

\begin{theorem}[{\cite[Theorem 3.1]{ABLG}}]\label{fact:CharUltraexacting}
    A cardinal $\lambda$ is ultraexacting if and only if for every ordinal definable subset $A$ of $V_{\lambda+1}$, there is a non-trivial elementary embedding $\map{j}{\langle V_{\lambda+1},\in,A\rangle}{\langle V_{\lambda+1},\in,A\rangle}$ with $\crit(j)<\lambda$. 
\end{theorem}

Later in \S\ref{sec: smallforcing}, we will discuss the effect of small forcings upon exacting and ultraexacting cardinals. The first basic observation in that respect is that the Levy--Solovay theorem \cite{LevySolovay} can be extended to these cardinals:

\begin{lemma}\label{lemma:ExactingSmallForcing}
 Let $\lambda$ be a cardinal and $\PPP\in H(\lambda)$ be a partial order. 
  \begin{enumerate}
      \item\label{item:ExactingUltraSmallForcing1} If $\lambda$ is an exacting cardinal, then $$\mathbbm{1}\Vdash_\PPP\textit{$``\check{\lambda}$ is an exacting cardinal }".$$ 

      \item\label{item:ExactingUltraSmallForcing2} If $\lambda$ is an ultraexacting cardinal, then $$\mathbbm{1}\Vdash_\PPP\textit{$``\check{\lambda}$ is an ultraexacting cardinal }".$$
  \end{enumerate}
\end{lemma}

\begin{proof}
  \eqref{item:ExactingUltraSmallForcing1} Pick $\lambda<\zeta\in C^{(3)}$ and an inaccessible cardinal $\delta<\lambda$ with $\PPP\in V_\delta$. We can now use our assumption to find an elementary submodel $X$ of $V_\zeta$ with $V_\lambda\cup\{\lambda\}\subseteq X$ and an elementary embedding $\map{j}{X}{V_\zeta}$ with $j\restriction \delta=\id_\delta$, $j(\lambda)=\lambda$ and $j\restriction\lambda\neq\id_\lambda$. 
   Let $G$ be $\PPP$-generic over $V$. Since $\mathcal{P}(\PPP)^V\subseteq X$, it follows that the set $$X[G] ~ = ~ \Set{\tau^G}{\textit{$\tau\in X$ is a $\PPP$-name}}$$ is an elementary submodel of $V[G]_\zeta$ and the embedding $j$ can be extended to $\map{j_G}{X[G]}{V[G]_\zeta}$. 
   Since the partial order $\PPP$ is an element of $V_\zeta$, it follows that $\zeta$ is an element of $(C^{(3)})^{V[G]}$ and, by Lemma \ref{lemma:CharExactingSingleEmb}, this shows that the embedding $j_G$ witnesses that $\lambda$ is an exacting cardinal in $V[G]$. 

   \eqref{item:ExactingUltraSmallForcing2} Pick $\delta$, $\zeta$, $X$ and $j$ satisfying the statements listed in \eqref{item:ExactingUltraSmallForcing1} and the additional assumption that $j\restriction V_\lambda$ is an element of $X$. Then $X$ contains a $\PPP$-name for the canonical lifting of $j\restriction V_\lambda$ to the $V_\lambda$ of $\PPP$-generic extensions. Combined with the argument in \eqref{item:ExactingUltraSmallForcing1}, this shows that $\PPP$-forces $\lambda$ to be an ultraexacting cardinal. 
\end{proof}


\subsection{Choiceless settings}\label{sec:LCnoAC}
In some parts of this paper, the ambient universe will only be a model of $\ZF$ and we need to adapt several of our central concepts to this more general setting. 
First, since properties that  characterize the same large cardinal property in $\ZFC$ can be inequivalent in choiceless models (see \cite{MR2310342}), we specify our definitions of  large cardinal notions in the $\ZF$-context following \cite{SEM1}:  
%

\begin{definition}[$\ZF$, {\cite[p. 323]{SEM1}}]\label{def:ZFinaccessible}
    A cardinal $\kappa$ is \emph{strongly inaccessible} if  for each  $\alpha<\kappa$, there is no  $\map{f}{V_\alpha}{\kappa}$ whose range is unbounded in $\kappa$. 
\end{definition}

\begin{definition}[$\ZF$, {{\cite[Definition 220]{SEM1})}}]\label{def:ZFsupercompact}
 A cardinal $\kappa$ is  \emph{supercompact} if for every $\eta>\kappa$, there is $\zeta>\eta$, a transitive set $N$ with ${}^{V_\eta}N\subseteq N$ and a non-trivial embedding $\map{j}{V_\zeta}{N}$ with $\crit(j)=\kappa$ and $j(\kappa)>\eta$.   
\end{definition}

  %

We will make use of these definitions  in \S\ref{sec: exacting from beyond choice} where we show how to get many exacting cardinals from large cardinals beyond choice.

\smallskip

Next, we consider the notions of exactingness and ultraexactingness in the absence of the Axiom of Choice. Following \cite{ABLG}, we use the characterizations presented in \S\ref{sec:Exacting} as definitions in this setting. It should be noted that the principle $\dc_\lambda$ suffices to prove these equivalences for a given infinite cardinal $\lambda$.

\begin{definition}[$\ZF$]
    A cardinal $\lambda$ is \emph{exacting} if for every non-empty ordinal definable subset $A$ of $V_{\lambda+1}$, there exist $x,y\in A$ and a non-trivial elementary embedding $\map{j}{\langle V_\lambda,\in,x\rangle}{\langle V_\lambda,\in,y\rangle}$ with $\crit(j)<\lambda$. 
\end{definition}

\begin{definition}[$\ZF$]\label{definition:ZFultraexacting} 
    A cardinal $\lambda$ is \emph{ultraexacting} if for every ordinal definable subset $A$ of $V_{\lambda+1}$, there is a non-trivial elementary embedding $\map{j}{\langle V_{\lambda+1},\in,A\rangle}{\langle V_{\lambda+1},\in,A\rangle}$ with $\crit(j)<\lambda$. 
\end{definition}

In the following, we discuss some basic structural properties of exacting and ultraexacting cardinals that will be used in our proofs. We start by discussing two settings in which these  properties are downward-absolute from the set-theoretic universe to  certain  models.

\begin{prop}[$\ZF$]\label{prop:ZFExactingDown}\label{prop:ExactingRankSegments}
  Let $\delta$ be an ordinal such that $V_\delta$ is a model of $\ZF$. 
   \begin{enumerate}
       \item If $\lambda<\delta$ is an exacting cardinal, then $\lambda$ is  exacting  in $V_\delta$. 

       \item If $\lambda<\delta$ is an ultraexacting cardinal, then $\lambda$ is  ultraexacting  in $V_\delta$. 
   \end{enumerate}
\end{prop}

\begin{proof}
    Both statements follow directly from the above definitions using the fact that every element in $\OD^{V_\delta}$ is ordinal definable in $V$. 
\end{proof}

\begin{prop}[$\ZF$]\label{prop:ExactingDownToL}
    Let $\lambda$ be a cardinal and let $\xi>\lambda$ be an ordinal. 
       \begin{enumerate}
        \item\label{prop:ZFexactL1} If $\lambda$ is an exacting cardinal, then $\lambda$ is an exacting cardinal in $L(V_\xi)$. 

        \item\label{prop:ZFexactL2} If $\lambda$ is an ultraexacting cardinal, then $\lambda$ is an ultraexacting cardinal in $L(V_\xi)$. 
    \end{enumerate}
\end{prop}

\begin{proof}
    \eqref{prop:ZFexactL1} Let $A$ be a subset of $V_{\lambda+1}$ that is ordinal definable in $L(V_\xi)$. 
    Then $A$ is ordinal definable in $V$ and there exist $x,y\in A$ with the property that there exists a non-trivial elementary embedding  of $\langle V_\lambda,\in,x\rangle$ into $\langle V_\lambda,\in,y\rangle$ in $V$ and this map is also an element of $L(V_\xi)$. Hence, $L(V_\xi)\models\lambda$ is exacting.
    
   \eqref{prop:ZFexactL2} Let $A$ be a subset of $V_{\lambda+1}$ that is ordinal definable in $L(V_\xi)$. Then $A$ is ordinal definable in $V$ and there exists a non-trivial elementary embedding $j$ of $\langle V_{\lambda+1},\in,A\rangle$ into itself. Then $j\restriction V_\lambda$ is an element of $L(V_\xi)$ and, since $\lambda$ is a limit ordinal, it follows that the map $j$ is also contained in $L(V_\xi)$. We can now conclude that $\lambda$ is an ultraexacting cardinal in $L(V_\xi)$. 
\end{proof}

For later use, we also establish an easy fact regarding the interplay between the above notions and set-theoretic forcings:

\begin{prop}[$\ZF$]\label{prop:PreserveExactingGoodForcing}
    Let $\lambda$ be a cardinal and let $\PPP$ be an ordinal definable,  weakly homogeneous partial order with the property that forcing with $\PPP$ does not add new subsets of $V_\lambda$. 
    \begin{enumerate}
        \item\label{prop:ZFforcingPreserve1} If $\lambda$ is an exacting cardinal, then $\lambda$ is an exacting cardinal in every $\PPP$-generic extension. 

        \item\label{prop:ZFforcingPreserve2} If $\lambda$ is an ultraexacting cardinal, then $\lambda$ is an ultraexacting cardinal in every $\PPP$-generic extension. 
    \end{enumerate}
\end{prop}

\begin{proof}
 \eqref{prop:ZFforcingPreserve1} Let $G$ be $\PPP$-generic over $V$ and let $A$ be a subset of $V[G]_{\lambda+1}$ that is ordinal definable in $V[G]$. Then $A$ consists of elements of $V$ and the weak homogeneity of $\PPP$ allows us to conclude that $A$ is an element of the ground model $V$. Since $\PPP$ is ordinal definable in $V$, we now know that $A$ is also ordinal definable in $V$. Moreover, the fact that $\lambda$ is an exacting cardinal in $V$ allows us to find $x,y\in A$ with the property that there exists a non-trivial elementary embedding of $\langle V_\lambda,\in,x\rangle$ into $\langle V_\lambda,\in,y\rangle$ in $V$. By our assumptions, this map is also a non-trivial elementary embedding of $\langle V[G]_\lambda,\in,x\rangle$ into $\langle V[G]_\lambda,\in,y\rangle$. These computations show that $\lambda$ is an exacting cardinal in $V[G]$.

 \eqref{prop:ZFforcingPreserve2} Let $G$ be $\PPP$-generic over $V$ and let $A$ be a subset of $V[G]_{\lambda+1}$ that is ordinal definable in $V[G]$. As above, it follows that $A$ is an ordinal definable element of the ground model $V$. Therefore, ultraexactingness yields a non-trivial elementary embedding of $\langle V_{\lambda+1},\in,A\rangle$ into itself. Since $V_{\lambda+1}=V[G]_{\lambda+1}$, this map is also a non-trivial elementary embedding of  $\langle V[G]_{\lambda+1},\in,A\rangle$ into itself. This shows that $\lambda$ is ultraexacting in $V[G]$. 
\end{proof}

In the remainder of this section, we refine ideas from consistency proofs in \cite{ABL} and \cite{ABLG} in order to isolate sufficient conditions for a cardinal to be exacting or ultraexacting in a transitive model of $\ZF$. 
This discussion will play a key role in the proofs of several of our main results. For example, we will use one of these criteria to show that (granted suitable large cardinal assumptions) the axioms of $\ZF$ are consistent with the existence of an exacting cardinal $\lambda$ whose successor is extendible in $\HOD$. 
 The following result extracts the core idea from the proof of   \cite[Theorem 5.4]{ABLG}:

\begin{theorem}[\ZF]\label{theorem:ExactingInnerModels}
  Let $\lambda$ be singular cardinal of countable cofinality and let $N$ be (possibly class-sized) transitive model of $\ZF$  such that $V_\lambda\in N$, $\cof{\lambda}^N=\omega$ and there is a well-ordering of $V_\lambda$ in $N$. 
  If there are subclasses $X_0$ and $X_1$ of $N$ that contain $V_\lambda\cup\{\lambda\}$ and are $\Sigma_2$-correct in $N$ and an elementary embedding $\map{j}{X_0}{X_1}$ with $j\restriction\lambda\neq\id_\lambda$ and $j(\lambda)=\lambda$, then $\lambda$ is an exacting cardinal in $N$. 
\end{theorem}

\begin{proof}
     Assume, towards a contradiction, that $\lambda$ is not an exacting cardinal in $N$. By definition, this means that, in $N$, there is a non-empty ordinal definable subset $A$ of $V_{\lambda+1}$ with the property that for all $x,y\in A$, there is no non-trivial elementary embedding of $\langle V_\lambda,\in,x\rangle$ into $\langle V_\lambda,\in,y\rangle$. 
     Using the $\Sigma_2$-definability of the class of proper initial segments of the canonical wellordering of ordinal definable sets (see {\cite[Proposition 3.9]{Sigma1Partitions}}), we can find a non-empty subset $A$ of $V_{\lambda+1}$ in $N$ with the property that, in $N$, the set $\{A\}$ is definable by a $\Sigma_2$-formula with parameter $\lambda$ and for all $x,y\in A$, there is no non-trivial elementary embedding of $\langle V_\lambda,\in,x\rangle$ into $\langle V_\lambda,\in,y\rangle$. 
     The $\Sigma_2$-correctness of $X$ in $N$ then ensures that $A$ is an element of $X_0\cap X_1$ with $A\cap X_0\neq\emptyset$ and $j(A)=A$. Pick $x\in A\cap X_0$ and set $y=j(x)\in A$. Since $V_\lambda$ is an element of $X_0$ with $j(V_\lambda)=V_\lambda$, it follows that $j\restriction V_\lambda$ is a non-trivial elementary embedding of $\langle V_\lambda,\in,x\rangle$ into $\langle V_\lambda,\in,y\rangle$.

     Now, fix a  sequence $\seq{\kappa_n}{n<\omega}$ in $N$ that is cofinal in $\lambda$ and satisfies $j(\kappa_0)>\kappa_0$. 
     Define $T$ to be the set of all partial non-trivial elementary embeddings $i$ of $\langle V_\lambda,\in,x\rangle$ into $\langle V_\lambda,\in,y\rangle$ such that $i$ is an element of $V_\lambda$, $\dom(i)=V_{\kappa_n}\cup\{\kappa_n\}$ for some $n<\omega$ and $i(\kappa_0)>\kappa_0$. 
     Then $T$ is an element of $N$ and, if we order $T$ by end-extensions of functions, then we turn this set into a tree of height at most $\omega$.  For each $n<\omega$, the restriction $j\restriction(V_{\kappa_n}\cup\{\kappa_n\})$ is an element of $T$. 
     It follows that $T$ has height $\omega$ and is ill-founded in $V$. 
     %
     %
     Fix a well-ordering $\lhd$ of $V_\lambda$ in $N$. Since $T$ is a subset of $V_\lambda$, the ill-foundedness of $T$ in $V$ ensures the existence of a unique $\lhd$-leftmost branch $b_\lhd$ through $T$. It then follows that  $b_\lhd$ is an element of $N$ and $\bigcup b$ is a non-trivial elementary embedding of $\langle V_\lambda,\in,x\rangle$ into $\langle V_\lambda,\in,y\rangle$ in $N$, contradicting our choice of $A$.  
\end{proof}

The following corollary follows from the above theorem by choosing $N=X_0=X_1$. It will be our main way to utilize the above result.

\begin{cor}[\ZF]\label{corollary:ExactingInnerModel}
    Let $\map{j}{V}{M}$ be a non-trivial elementary embedding with first non-trivial fixed point $\lambda$. Assume that $N\subseteq M$ is a (possibly class-sized) transitive model of $\ZF$ such that $V_\lambda\in N$, $\cof{\lambda}^N=\omega$ and there is a well-ordering of $V_\lambda$ in $N$.    
    If   $j\restriction N$ is an elementary embedding of $N$ into $N$, then $\lambda$ is an exacting cardinal in $N$.  \qed
\end{cor}

The previous corollary  has the following  $\ZFC$-application:

\begin{cor}\label{corollary:ExactingModels}
    Let $\map{j}{V}{M}$ be an I2-embedding with critical sequence  $\vec{\kappa}=\seq{\kappa_m}{m<\omega}$ and first non-trivial fixed point $\lambda$ that is given by a $(\kappa_0,\lambda)$-extender $E$. Let $U$ denote the ultrafilter on $\kappa$ induced by $j$ and let $\seq{M_\alpha}{\alpha\in\ord}$ denote  the sequence of inner models obtained by iterating $V$ and $E$. 
    \begin{enumerate}
        \item\label{item:ExactingModels1} The sequence $\vec{\kappa}$ is generic over $M_\omega$ and $\lambda$ is an exacting cardinal in $M_\omega[\vec{\kappa}]$. 

        \item\label{item:ExactingModels2} If $\PPP_U$ is the Prikry forcing using $U$, then $\PPP_U$ forces $\kappa$ to be an exacting cardinal. 

        \item\label{item:ExactingModels3} The class $\bigcap_{n<\omega} M_n$ is a model of $\ZF$ and $\lambda$ is an exacting cardinal in $\bigcap_{n<\omega} M_n$. 
    \end{enumerate} 
\end{cor}

\begin{proof}
  \eqref{item:ExactingModels1} Standard arguments show that the sequence $\vec{x}$ induces a generic filter for the Prikry forcing in $M_\omega$ that uses the image of $U$ under the canonical embedding of $V$ into $M_\omega$. Moreover, it is easy to show that $j\restriction M_\omega$ is an elementary embedding of $M_\omega$ into itself. Since $j(\vec{\kappa})=\seq{\kappa_{m+1}}{m<\omega}$ and therefore $M_\omega[\vec{\kappa}]=M_\omega[j(\vec{\kappa})]$, it follows that $j\restriction M_\omega[\vec{\kappa}]$ is an elementary embedding of $M_\omega[\vec{\kappa}]$ into itself. Finally, since $M_\omega[\vec{\kappa}]$ is an inner model of $\ZFC$ containing $V_\lambda$ and $\vec{\kappa}$ is cofinal in $\lambda$, we can apply Corollary \ref{corollary:ExactingInnerModel} to conclude that $\lambda$ is an exacting cardinal in $M_\omega[\vec{\kappa}]$. 

  \eqref{item:ExactingModels2} Using the elementarity of the canonical embedding of $V$ into $M_\omega$, the given statement follows directly from the above proof of \eqref{item:ExactingModels1}.
  
  \eqref{item:ExactingModels3}  Set $N=\bigcap_{n<\omega} M_n$. Standard arguments then show that $N$ is a model of $\ZF$ that contains $M_\omega$. In particular, it follows that $V_\lambda$ is contained in $N$ and there is a well-ordering of $V_\lambda$ in $N$. Moreover, it is easy to see that the sequence $\vec{\kappa}$ is an element of $N$ and therefore $\lambda$ has countable cofinality in $N$. The desired conclusion now follows from Corollary \ref{corollary:ExactingInnerModel}.  
\end{proof}

We end this section by proving an analog of Theorem \ref{theorem:ExactingInnerModels} for ultraexactingness:

\begin{theorem}[\ZF]\label{theorem:UltraexactingInnerModels}
  Let $\lambda$ be singular cardinal and let $N$ be (possibly class-sized) transitive model of $\ZF$  with the property that  $V_\lambda\in N$ and $\lambda$ is singular in $N$. 
  If there are subclasses $X_0$ and $X_1$ of $N$ that contain $V_\lambda\cup\{\lambda\}$ and are $\Sigma_2$-correct in $N$ and an elementary embedding $\map{j}{X_0}{X_1}$ with $j\restriction\lambda\neq\id_\lambda$, $j\restriction\cof{\lambda}^N=\id_{\cof{\lambda}^N}$,  $j(\lambda)=\lambda$ and $j\restriction V_\lambda\in X_0$, then $\lambda$ is an ultraexacting cardinal in $N$. 
\end{theorem}

\begin{proof}
  Assume, towards a contradition, that $\lambda$ is not an ultraexacting cardinal in $N$. Then, in $N$, there is an ordinal definable subset $A$ of $V_{\lambda+1}$ with the property that  there is no non-trivial elementary embedding of $\langle V_\lambda,\in,A\rangle$ into itself. 
     By using the canonical well-ordering of ordinal definable sets, we can find a subset $A$ of $V_{\lambda+1}$ in $N$ with the property that, in $N$, the set $\{A\}$ is definable by a $\Sigma_2$-formula with parameter $\lambda$ and  there is no non-trivial elementary embedding of $\langle V_\lambda,\in,A\rangle$ into itself. 
     The correctness properties of $X_0$ and $X_1$  now imply that $A\in X_0\cap X_1$ and $j(A)=A$. 

     Next, we can find  a  sequence $\seq{\kappa_\gamma}{\gamma<\cof{\lambda}^N}$ in $X_0$ that is cofinal in $\lambda$. Define $$\map{i}{V_{\lambda+1}^N}{V_{\lambda+1}^N}; ~ x\mapsto\bigcup\Set{(j\restriction V_\lambda)(x\cap V_{\kappa_\gamma})}{\gamma<\cof{\lambda}^N}.$$ Then $i$ is an element of $X_0$ and therefore $i[V_{\lambda+1}\cap X_0]\subseteq X_0$. Moreover, since $j(\lambda)=\lambda$, it follows that $j(\cof{\lambda}^N)=\cof{\lambda}^N$ and therefore we have   
     $$j\restriction(\cof{\lambda}^N+1) ~ = ~ \id_{\cof{\lambda}^N+1}.$$ Since this directly implies that $$i\restriction(V_{\lambda+1}\cap X_0) ~ = ~ j\restriction(V_{\lambda+1}\cap X_0),$$ we can conclude that $i$ is a non-trivial elementary embedding of $\langle V_\lambda,\in,A\rangle$ into itself in $X_0$. The correctness properties of $X_0$ in $N$ now ensure that $i$ is a function with these properties in $N$, contradicting our choice of $A$. 
\end{proof}


\subsection{Large cardinals beyond choice}\label{section:LCbeyondChoice}  
 We now review definitions from the theory of large cardinal beyond choice, as initiated in \cite{BagKoelWoo}. In addition, we some implications between these notions that will be need in \S\ref{sec: exacting from beyond choice}. Several of these large cardinal properties are defined through second-order statements. Following \cite{BagKoelWoo}, we will work with these notions in the theory  $\ZF_2$, the second-order version of $\ZF$ that extends the Separation, Collection and Replacement schemes to second-order formulas. Note that the axioms of $\ZF$ prove that, if $\kappa$ is a strongly inaccessible cardinal (in the sense of Definition \ref{def:ZFinaccessible}), then $\langle V_\kappa,V_{\kappa+1},\in\rangle$ is a model of $\ZF_2$. The following notion is the initial example of a large cardinal beyond choice:

\begin{definition}[$\ZF_2$]
   A \emph{Reinhardt embedding} is a non-trivial elementary embedding $\map{j}{V}{V}$. A \emph{Reinhardt cardinal} is the critical point of a Reinhardt embedding. 
\end{definition}

The following weakening of this notion was introduced by Goldberg and Schlutzenberg in \cite{GoSch24}. Many interesting consequences of the existence of Reinhardt cardinals factor through this property. It has the advantage of having a first-order formulation.

\begin{definition}[$\ZF$, {\cite{GoSch24}}] 
 A cardinal $\lambda$ is \emph{rank-Berkeley} if for all ordinals $\alpha<\lambda<\zeta$, there is a non-trivial elementary embedding $\map{j}{V_\zeta}{V_\zeta}$ such that $\alpha<\crit(j)<\lambda$ and $\lambda$ is the first non-trivial fixed point of $j$. 
\end{definition}

A short argument shows that the first non-trivial fixed point of a Rainhardt embedding is a rank-Berkeley cardinal.

\begin{prop}\label{prop:LocalRankBerkeleySupercompact}
    If $\delta$ is a supercompact cardinal and $\lambda<\delta$ is rank-Berkeley in $V_\delta$, then $\lambda$ is rank-Berkeley. 
\end{prop}

\begin{proof}
  Fix an ordinal $\eta\geq\delta$. Then there is $\zeta>\eta$, a transitive set $N$ with ${}^{V_\eta}N\subseteq N$ and a non-trivial elementary embedding $\map{j}{V_\zeta}{N}$ with $\crit(j)=\delta$ and $j(\delta)>\eta$. Since $\lambda$ is rank-Berkeley in $j(V_\delta)$ and $V_\eta\in N$, there exists an elementary embedding $\map{i}{V_\eta}{V_\eta}$ with critical point below $\lambda$ and first non-trivial fixed point $\lambda$.  
\end{proof}

Next, we consider a global strengthening of Reinhardtness:

\begin{definition}[$\ZF_2$, {\cite{BagKoelWoo}}]
 Given a proper class $A$, a cardinal $\kappa$ is  \emph{$A$-super Reinhardt} if for every $\lambda>\kappa$, there is a Reinhardt embedding $\map{j}{V}{V}$ satisfying $\crit(j)=\kappa$, $j(\kappa)>\lambda$ and $j^+(A)=A$, where $$j^+(A) ~ = ~ \bigcup_{\alpha\in\ord} j(A\cap V_\alpha).$$ 
A cardinal is  \emph{super Reinhardt} if it is $V$-super Reinhardt. 
\end{definition}

\begin{definition}[$\ZF$, {\cite{BagKoelWoo}}]
 A cardinal $\kappa$ is  \emph{totally Reinhardt} if $$\langle V_\kappa,V_{\kappa+1},\in\rangle\models \ZF_2  +  \textit{``There exists an $A$-super Reinhardt cardinal"}$$ holds for every  $A\in V_{\kappa+1}\setminus A$. 
\end{definition}

\begin{lemma}[$\ZF_2$]\label{lemma:SuperReinhardtExacting}
    If $\delta$ is a super Reinhardt cardinal, then $\delta$ is a limit of  rank-Berkeley cardinals $\zeta<\delta$ that are limits of rank-Berkeley cardinals $\lambda$ with $V_\lambda\prec V_\delta$. 
\end{lemma}

\begin{proof}
 Fix $\alpha<\delta$ and pick an elementary embedding  $\map{j}{V}{V}$ with critical point $\delta$. Let $\eta$ denote the first non-trivial fixed point  of $j$. Since $V_\delta\prec V$, it follows that $\eta$ is a rank-Berkeley cardinal with $V_\eta\prec V$. Then there exists an elementary embedding $\map{i}{V}{V}$ with $\crit(i)=\delta$ and $i(\delta)>\eta$. We now have $V_\eta\prec V_{i(\delta)}$. It follows that $\delta$ is a limit of rank-Berkeley cardinals $\lambda$ with $V_\lambda\prec V_\delta$. The elementarity of $j$ then ensures that $\eta$ is a limit of rank-Berkeley cardinals $\lambda$ with $V_\lambda\prec V_\eta$. Since $V_\eta\prec V_{i(\delta)}$, we can now use the elementarity of $i$ to find a rank-Berkeley cardinal $\alpha<\zeta<\delta$ that is a limit of rank-Berkeley cardinals $\lambda$ with $V_\lambda\prec V_\delta$.  
\end{proof}



\subsection{Magidor products of Prikry forcings}

In the following, we review arguments due to Magidor showing that a product of Prikry forcings on a sufficiently sparse set of measurable cardinals can be constructed so as to ensure that the resulting partial order possesses the Prikry property. This construction will later allow us to prove Theorem \ref{thm: a proper class
of exactings}.

Let us fix a discrete\footnote{In the sense that $D\cap\delta$ is a bounded
subset of $\delta$ for every cardinal $\delta$ in $D$.} set $D$ of measurable
cardinals and a sequence $\seq{\calU_\delta}{\delta\in D}$ with the property
that $\calU_\delta$ is a normal ultrafilter on $\delta$ for all $\delta\in D$.
Given $\delta\in D$, we let $\PPP_{\calU_\delta}$ denote the Prikry forcing using $\calU_\delta$.

\begin{definition}\label{def:Magidorproduct2}\hfill
 \begin{enumerate}
     \item Given $E\subseteq D$, a condition in the \emph{Magidor product}
     $$\PPP_E ~ = ~ \prod^{\mathrm{Mag}}_{\delta\in E}\PPP_{\calU_\delta}$$
     of the sequence $\seq{\PPP_{\calU_\delta}}{\delta\in E}$ is a sequence
     $\seq{p_\delta}{\delta\in E}$ with the property that $p_\delta$ is a
     condition in $\PPP_{\calU_\delta}$ for all $\delta\in E$ and the stem of
     $p_\delta$ is empty for all but finitely-many $\delta\in E$.

     \item Given conditions $\vec{p}=\seq{p_\delta}{\delta\in E}$ and
     $\vec{q}=\seq{q_\delta}{\delta\in E}$ in $\PPP_E$, we define
     $\vec{p}\leq_{\PPP_E}\vec{q}$ to hold if
     $p_\delta\leq_{\PPP_{\calU_\delta}}q_\delta$ holds for all $\delta\in E$ and
     $p_\delta\leq_{\PPP_{\calU_\delta}}^*q_\delta$ holds for all but finitely-many
     $\delta\in E$.
     In addition, we define $\vec{p}\leq_{\PPP_E}^*\vec{q}$ to hold if
     $p_\delta\leq_{\PPP_{\calU_\delta}}^*q_\delta$ holds for all $\delta\in E$.
 \end{enumerate}
\end{definition}

In particular, given $\delta\in D$, the above definition provides the partial
orders $\PPP_{D\cap\delta}$, $\PPP_{D\setminus\delta}$ and
$\PPP_{D\setminus(\delta+1)}$, as well as $\PPP_D$ itself.
Similar arguments to those utilized by Magidor in \cite{HowlargeMagidor} show that
$\langle\PPP_D,\leq_{\PPP_D},\leq_{\PPP_D}^*\rangle$ is of Prikry-type.

\begin{lemma}\label{MagProdProps}
The following statements hold for all $\delta\in D$:
\begin{enumerate}
  \item\label{MagProdProp21} The partial order $\PPP_D$ is isomorphic to
  $$
    \PPP_{D\cap\delta}\times\PPP_\delta\times\PPP_{D\setminus(\delta+1)},
  $$
  and $\PPP_{D\setminus\delta}$ is isomorphic to
  $\PPP_\delta\times\PPP_{D\setminus(\delta+1)}$. Both isomorphisms also
  preserve the direct extension orders.
 
  \item\label{MagProdProp22} In every $\PPP_{D\cap\delta}$-generic extension
  $V[G]$ of the ground model $V$, the partial order $\PPP_{D\setminus\delta}$
  is weakly homogeneous, $\langle\PPP_{D\setminus\delta},
  \leq_{\PPP_{D\setminus\delta}},\leq_{\PPP_{D\setminus\delta}}^*\rangle$ is
  of Prikry-type and $\leq_{\PPP_{D\setminus\delta}}^*$ is
  ${<}\delta$-directed-closed.
 
  \item\label{MagProdProp23} In every $\PPP_{D\cap\delta}$-generic extension
  $V[G]$ of the ground model $V$, forcing with $\PPP_{D\setminus\delta}$ adds
  no new subsets of ordinals less than $\delta$. In particular, if $\xi$ is an
  ordinal with $|V_\xi|<\delta$, then $V_\xi$ is computed the same way in
  $V[G]$ and in every $\PPP_{D\setminus\delta}$-generic extension of $V[G]$.
\end{enumerate}
\end{lemma}
 
\begin{proof}
\eqref{MagProdProp21} The natural decomposition map serves as an isomorphism
in both cases: splitting a condition according to the given partition of $D$
preserves both $\leq$ and $\leq^*$, and the requirement that only finitely
many coordinates carry a non-empty stem is equivalent to the conjunction of
the corresponding requirements for the factors.
 
\smallskip
 
\eqref{MagProdProp22} Set $\sigma=\sup(D\cap\delta)$, so that $\sigma<\delta$
because $D$ is discrete. Two conditions of $\PPP_{D\cap\delta}$ with the same
stems are compatible, and the number of possible sequences of stems is at most
$\sigma$. Consequently, $\PPP_{D\cap\delta}$ satisfies the $\sigma^+$-chain
condition in $V$. Moreover $|\PPP_{D\cap\delta}|\leq 2^\sigma<\delta$, because
$\delta$ is inaccessible. Let $G$ be $\PPP_{D\cap\delta}$-generic over $V$.
 
Since $\PPP_{\calU_\eta}$ is weakly homogeneous in $V$ for every $\eta\in D$, and weak
homogeneity is preserved by Magidor products, the partial order
$\PPP_{D\setminus\delta}$ is weakly homogeneous in $V$. As the isomorphisms
witnessing this are elements of $V$ and compatibility of conditions is
absolute, it follows that $\PPP_{D\setminus\delta}$ is also weakly homogeneous
in $V[G]$.
 
Next, since $|\PPP_{D\cap\delta}|<\delta\leq\eta$ for every
$\eta\in D\setminus\delta$, the Levy--Solovay theorem implies that
$$
  \bar{\calU}_\eta ~ = ~
  \Set{X\in\mathcal{P}^{V[G]}(\eta)}{\exists A\in\calU_\eta ~ A\subseteq X}
$$
is a normal ultrafilter on $\eta$ in $V[G]$. Working in $V[G]$, define
$$
  \PPP^G_{D\setminus\delta} ~ = ~
  \prod_{\eta\in D\setminus\delta}^{\mathrm{Mag}}\PPP_{\bar{\calU}_\eta}.
$$
Results of Magidor in \cite{HowlargeMagidor} then show that
$\langle\PPP^G_{D\setminus\delta},\leq_{\PPP^G_{D\setminus\delta}},
\leq^*_{\PPP^G_{D\setminus\delta}}\rangle$ is of Prikry-type in $V[G]$.
Moreover, the ordering $\leq^*_{\PPP^G_{D\setminus\delta}}$ is ${<}\delta$-directed-closed
in $V[G]$: the members of a directed family all have the same stems,
and, since each $\bar{\calU}_\eta$ is ${<}\eta$-complete in $V[G]$ and
$\delta\leq\eta$, fewer than $\delta$-many of the corresponding measure one
sets can be intersected coordinatewise.
 
Both properties are inherited by $\leq^*$-dense subposets. Indeed, assume that
$\PPP_{D\setminus\delta}$ is a $\leq^*$-dense subposet of
$\PPP^G_{D\setminus\delta}$, {i.e.,}\ that every condition of
$\PPP^G_{D\setminus\delta}$ has a $\leq^*$-stronger condition in
$\PPP_{D\setminus\delta}$. Since $\leq^*$ refines $\leq$, the partial order
$\PPP_{D\setminus\delta}$ is then also dense in $\PPP^G_{D\setminus\delta}$,
so both partial orders have the same generic extensions and the same forcing
relation. Given a statement $\varphi$ of the forcing language and a condition
$\vec{p}$ in $\PPP_{D\setminus\delta}$, the Prikry property of
$\PPP^G_{D\setminus\delta}$ provides $\vec{p}^{\,*}\leq^*\vec{p}$ in
$\PPP^G_{D\setminus\delta}$ deciding $\varphi$, and $\leq^*$-density yields
$\vec{q}\leq^*\vec{p}^{\,*}$ in $\PPP_{D\setminus\delta}$, which then also
decides $\varphi$. Similarly, a $\leq^*$-directed family of size less than
$\delta$ in $\PPP_{D\setminus\delta}$ has a $\leq^*$-lower bound in
$\PPP^G_{D\setminus\delta}$, and any condition of $\PPP_{D\setminus\delta}$
that is $\leq^*$-stronger than this bound is a $\leq^*$-lower bound of the
family in $\PPP_{D\setminus\delta}$.
 
It therefore suffices to show that $\PPP_{D\setminus\delta}$ is a
$\leq^*$-dense subposet of $\PPP^G_{D\setminus\delta}$. Let
$\vec{p}=\seq{p_\eta}{\eta\in D\setminus\delta}$ be a condition in
$\PPP^G_{D\setminus\delta}$ and, given $\eta\in D\setminus\delta$, let
$s_\eta$ denote the stem of $p_\eta$ and let $A_\eta\in\bar{\calU}_\eta$
denote the second component of $p_\eta$. By the definition of
$\bar{\calU}_\eta$, working in $V[G]$ we may pick a sequence
$\seq{B_\eta\in\calU_\eta}{\eta\in D\setminus\delta}$ with the property that
$B_\eta\subseteq A_\eta$ holds for all $\eta\in D\setminus\delta$.
 
Fix a $\PPP_{D\cap\delta}$-name $\dot{B}$ for the sequence
$\seq{B_\eta}{\eta\in D\setminus\delta}$ and, in $V$, set
$$
  \calC_\eta ~ = ~
  \Set{C\in\calU_\eta}{\exists p\in\PPP_{D\cap\delta} ~
  p\Vdash\dot{B}(\check{\eta})=\check{C}}.
$$
Conditions forcing distinct values of $\dot{B}(\check\eta)$ are incompatible,
so the $\sigma^+$-chain condition ensures that $|\calC_\eta|\leq\sigma$ holds
for all $\eta\in D\setminus\delta$; hence the sequence
$\seq{\calC_\eta\in[\calU_\eta]^{\leq\sigma}}{\eta\in D\setminus\delta}$
belongs to $V$ and $B_\eta\in\calC_\eta$ for all $\eta\in D\setminus\delta$.
 
By the definition of the Magidor product, the set
$\Set{\eta\in D\setminus\delta}{s_\eta\neq\emptyset}$ is finite, and each
$s_\eta$ is a finite sequence of ordinals; in particular, it follows that the
sequence $\seq{s_\eta}{\eta\in D\setminus\delta}$ is an element of $V$. We
then know that $V$ contains the sequence
$\vec{q}=\seq{q_\eta}{\eta\in D\setminus\delta}$ with the property that
$$
  q_\eta ~ = ~ \bigl\langle s_\eta,\textstyle\bigcap\calC_\eta\bigr\rangle
$$
holds for all $\eta\in D\setminus\delta$. Since $\calU_\eta$ is
${<}\eta$-complete in $V$ and $|\calC_\eta|\leq\sigma<\delta\leq\eta$, we have
$\bigcap\calC_\eta\in\calU_\eta$; after removing from $\bigcap\calC_\eta$ its
finitely many elements below $\max(s_\eta)$, which affects only the finitely
many coordinates with $s_\eta\neq\emptyset$, it follows that $\vec{q}$ is a
condition in $\PPP_{D\setminus\delta}$ in $V$ and hence also a condition in
$\PPP^G_{D\setminus\delta}$ in $V[G]$. Finally, $B_\eta\in\calC_\eta$ implies
$\bigcap\calC_\eta\subseteq B_\eta\subseteq A_\eta$ for all
$\eta\in D\setminus\delta$, and $\vec{q}$ and $\vec{p}$ have the same stems.
Therefore $\vec{q}$ is $\leq^*$-stronger than $\vec{p}$ in $V[G]$.
 
\smallskip
 
\eqref{MagProdProp23} Let $\rho<\delta$, let $\vec{p}$ be a condition in
$\PPP_{D\setminus\delta}$ and let $\dot{X}$ be a
$\PPP_{D\setminus\delta}$-name for a subset of $\rho$. Using
\eqref{MagProdProp22}, recursively construct a $\leq^*$-decreasing sequence
$\seq{\vec{p}_\alpha}{\alpha\leq\rho}$ of conditions below $\vec{p}$ such that
$\vec{p}_{\alpha+1}$ decides the statement ``$\check{\alpha}\in\dot{X}$'',
using the Prikry property at successor steps and the
${<}\delta$-directed-closure of $\leq^*$ at limit steps. Then $\vec{p}_\rho$
decides all statements ``$\check{\alpha}\in\dot{X}$'' with $\alpha<\rho$ and
therefore forces $\dot{X}$ to be equal to a set in $V[G]$. The second
assertion follows by induction on $\xi$, coding the elements of $V_\xi$ by
subsets of $|V_\xi|<\delta$.
\end{proof}

\begin{cor}\label{corollary:MagidorSupportPreserveInaccessible}
   Let $\nu=\sup(D)$ and assume that $\nu$ is an
   inaccessible cardinal. Then forcing with $\PPP_D$ preserves the
   inaccessibility of $\nu$.
\end{cor}
 
\begin{proof}
  Let $G$ be $\PPP_D$-generic over $V$ and let $\rho<\nu$. Fix $\delta\in D$
  with $\rho<\delta$ and let $G_{{<}\delta}\times G_{{\geq}\delta}$ denote the
  filter on $\PPP_{D\cap\delta}\times\PPP_{D\setminus\delta}$ induced by $G$
  through Lemma~\ref{MagProdProps}.\eqref{MagProdProp21}. By
  Lemma~\ref{MagProdProps}.\eqref{MagProdProp23}, we have
  $\mathcal{P}(\rho)^{V[G]}=\mathcal{P}(\rho)^{V[G_{{<}\delta}]}$ and, since
  $|\PPP_{D\cap\delta}|<\delta<\nu$ and $\nu$ is inaccessible in $V$, this set
  has size less than $\nu$. Hence $\nu$ is a strong limit cardinal in $V[G]$.

  The regularity of $\nu$ in $V[G]$ follows similarly.
\end{proof}


\subsection{Supercompact Prikry forcing}\label{sec: Supercompact Prikry}
In this section, we review the classical \emph{Supercompact Prikry forcing}  with an eye toward proving   Theorem~\ref{theorem:ExactingInHOD}. 

Suppose that $\kappa<\lambda$ are cardinals with $\kappa$  supercompact  and $\lambda$ inaccessible. Fix   a normal, fine, ${<}\kappa$-complete ultrafilter $\mathcal{U}$ on $\mathcal{P}_\kappa(\lambda)$. Given $s,t\in \mathcal{P}_\kappa(\lambda)$, we write $s\prec t$ whenever $s\s t$ and $\otp(s)<\otp(t\cap \kappa)$.  A condition in the \emph{Supercompact Prikry forcing with respect to $\mathcal{U}$}, $\mathbb{P}_{\mathcal{U}}$, is a pair $\langle s, A\rangle$ such that $s=\langle s_0,\dots, s_{n-1}\rangle$ is a $\prec$-increasing sequence of members of $\mathcal{P}_\kappa(\lambda)$, $A\in \mathcal{U}$ and $s_{n-1}\prec x$ for all $x\in A$. Given conditions $\langle s,A\rangle$ and $\langle t, B\rangle$ in $\PPP_\calU$, we write $\langle s, A\rangle\leq \langle t, B\rangle$ whenever $t\sqsubseteq s$, $A\s B$ and $t\setminus s\s B$, and $\langle s, A\rangle\leq^*\langle t,B\rangle$ whenever  $\langle s, A\rangle\leq\langle t,B\rangle$ and $s=t$. 
If $G$ is a $\mathbb{P}_{\mathcal{U}}$-generic filter, then   the sequence $\vec{x}_G=\langle x_n^G\mid n<\omega\rangle$ given by $x_n^G=s^p_n$ for some (equivalently, every) condition $p\in G$ is $\prec$-increasing and it satisfies that the filter 
$$G(\vec{x}_G) ~ = ~ \Set{\langle s, A\rangle}{s\sq \vec{x}_G, ~ \vec{x}_G\setminus s\s A}$$
is equal to $G$. In particular, we have $V[G]=V[\vec{x}_G]$ in this case.

The following lemma list the basic properties of $\PPP_\calU$. The proofs of these statements can be found in {\cite[\S1]{Gitik-handbook}}. 
\begin{lemma}
    \begin{enumerate}
        \item $\langle \mathbb{P}_{\mathcal{U}},\leq,\leq^*\rangle$ is a Prikry-type forcing. 
        
        \item $\mathbb{P}_{\mathcal{U}}$ preserves cardinals $\leq \kappa$.
        \item $\mathbb{P}_{\mathcal{U}}$ forces $``\cof{\kappa}=\omega$".
        \item Forcing with $\mathbb{P}_{\mathcal{U}}$ collapses all cardinals $\theta\in (\kappa,\mu]$, where $\mu=\lambda$ whenever $\cof{\lambda}\geq \kappa$ and $\mu=\lambda^+$ whenever $\cof{\lambda}<\kappa$. \qed
    \end{enumerate}
\end{lemma}

 For each regular cardinal $\alpha\in [\kappa,\lambda)$, let $\map{\pi_{\alpha}}{\mathcal{P}_\kappa(\lambda)}{\mathcal{P}_\kappa(\alpha)}$ be the projection map given by $x\mapsto x\cap \alpha$ and $\mathcal{U}_\alpha$  the Rudin--Keisler image of $\mathcal{U}$ via $\pi_\alpha$.\footnote{That is, $\mathcal{U}_\alpha =\{X\s \mathcal{P}_\kappa(\alpha)\mid (\pi_\alpha)^{-1}[X]\in \mathcal{U}\}.$} Suppose that $\vec{x}$ is $\mathbb{P}_{\mathcal{U}}$-generic. Then,  the \emph{Mathias criterion of genericity} for $\mathbb{P}_{\mathcal{U}}$ (see \cite[\S1]{Gitik-handbook}) ensures that the sequence $$\vec{x}_\alpha ~ = ~ \seq{x_n\cap \alpha}{n<\omega}$$  is generic for $\mathbb{P}_{\mathcal{U}_\alpha}$.

\section{A proper class of exacting cardinals from an I2-cardinal}\label{section:ProperClassExacting}

In \cite{ABLG}, it is shown that Prikry forcing with the critical point of an
I2-embedding makes the cardinal exacting in the generic extension (see Corollary \ref{corollary:ExactingModels}.\eqref{item:ExactingModels2}). In this
section, we leverage this observation to prove Theorem~\ref{thm: a proper class
of exactings}. The basic idea is to replace Prikry forcing by a Magidor product
of Prikry forcings in the argument from \cite{ABLG} and then cut off the
universe at an appropriate cardinal.
Later, we will show that
this naïve approach is somewhat optimal, in the sense that it cannot be
upgraded to produce a model with an exacting cardinal that is a limit of
exactings.


\subsection{The consistency of many exacting cardinals}
In the following, we let $\map{j}{V}{M}$ denote an I2-embedding, {i.e.,} $j$ is
a non-trivial elementary embedding with $V_\lambda\subseteq M$, where $\lambda$
is the first non-trivial fixed point of $j$.
We set $\kappa=\crit(j)$, we let $\seq{\kappa_n}{n<\omega}$ denote the critical
sequence of $j$ and we let
$$\calU ~ = ~ \Set{A\subseteq\kappa}{\kappa\in j(A)}$$
denote the normal ultrafilter on $\kappa$ derived from $j$.


\begin{lemma}\label{lemma:BelowKappa}
  If $A$ denotes the set of all cardinals $\rho<\kappa$ with the property that
  there exists a normal ultrafilter $\calF$ on $\rho$ such that $\rho$ is an
  exacting cardinal in $V_\lambda[H]$ whenever $H$ is $\PPP_\calF$-generic
  over $V_\lambda$, then $A\in\calU$.
\end{lemma}
 
\begin{proof}
 Let $H$ be $\PPP_{\calU}$-generic over $V_\lambda$. Every dense subset of
 $\PPP_\calU$ in $V$ has rank less than $\lambda$ and is therefore an element
 of $V_\lambda$, so $H$ is also $\PPP_{\calU}$-generic over $V$ and
 Corollary \ref{corollary:ExactingModels}.\eqref{item:ExactingModels2}  shows that $\kappa$ is an exacting cardinal in
 $V[H]$. Since $\PPP_\calU$ is an element of $V_\lambda$, we have
 $V[H]_\lambda=V_\lambda[H]$, and this is a model of $\ZFC$ because
 $V_{\kappa_n}$ is an elementary submodel of $V_\lambda$ for all $n<\omega$.
 Proposition~\ref{prop:ExactingRankSegments} therefore shows that $\kappa$ is also
 an exacting cardinal in $V_\lambda[H]$.
 Since $j(\lambda)=\lambda$ and $V_\lambda\subseteq M$ imply
 $j(V_\lambda)=V_\lambda$, the set $j(A)$ is the set of all cardinals
 $\rho<\kappa_1$ for which there exists a normal ultrafilter $\calF$ on $\rho$
 in $M$ such that $\rho$ is an exacting cardinal in $V_\lambda[H]$ whenever
 $H$ is $\PPP_\calF$-generic over $V_\lambda$. As $\calU$ is an element of
 $V_\lambda\subseteq M$ and all of the above computations take place in
 $V_\lambda$ and in its generic extensions, and are therefore performed the
 same way in $V$ and in $M$, these observations show that $\kappa$ is an
 element of $j(A)$ and this means that $A$ is an element of $\calU$.
\end{proof}

Now, let $A$ denote the subset of $\kappa$ given by
Lemma~\ref{lemma:BelowKappa} and fix a sequence $\seq{\calU_\rho}{\rho\in A}$
with the property that, for all $\rho\in A$, the set $\calU_\rho$ is a normal
ultrafilter on $\rho$ such that forcing with $\PPP_{\calU_\rho}$ over
$V_\lambda$ turns $\rho$ into an exacting cardinal.
   Set $D=A\setminus\Lim(A)$. Since $A$ is an element of $\calU$, the set $A$ is unbounded
   in $\kappa$ and hence so is $D$. Then $D$ is a discrete set of measurable
   cardinals with $\sup D=\kappa$, and we let $\PPP_D$ denote the Magidor
   product of Prikry forcings given by the sequence
   $\seq{\calU_\delta}{\delta\in D}$, as in
   Definition~\ref{def:Magidorproduct2}. Note that $\PPP_D$ is an element of
   $V_\lambda$, so that $V_\lambda[G]=V[G]_\lambda$ holds for every
   $\PPP_D$-generic filter $G$.

 \smallskip

\begin{lemma}\label{lemma:ManyExacting}
  Assume that  $V_\lambda$ is a model of $V=\gHOD$.\footnote{Recall that, by Lemma~\ref{lemma:I2andVHOD}, this assumption is compatible with the existence of an I2-cardinal.}
  If $G$ is $\PPP_D$-generic over $V$ and $\delta\in D$, then $\delta$ is an
  exacting cardinal in $V_\lambda[G]$.
\end{lemma}
 
\begin{proof}
  Fix $\delta\in D$ and let $\delta_*=\min(D\setminus(\delta+1))$ denote the
  least element of $D$ above $\delta$. Since $\delta_*$ is inaccessible and
  greater than $\delta$, we have $|V_{\delta+\omega}|<\delta_*$.
  Let $G_{{<}\delta}\times g\times G_{{>}\delta}$ be the filter on
  $\PPP_{D\cap\delta}\times\PPP_\delta\times\PPP_{D\setminus(\delta+1)}$
  induced by $G$ through the isomorphism provided by
  Lemma~\ref{MagProdProps}.\eqref{MagProdProp21} and set
  $W=V_\lambda[G_{{<}\delta}\times g]$, so that $V_\lambda[G]$ is a
  $\PPP_{D\setminus(\delta+1)}$-generic extension of $W$.
  Our
  setup new ensures that $\delta$ is an exacting cardinal in $V_\lambda[g]$.
  Moreover, since $D$ is discrete, it follows that the partial order
  $\PPP_{D\cap\delta}$ is an element of $H(\delta)^{V_\lambda[g]}$ and we can
  use Lemma~\ref{lemma:ExactingSmallForcing} in $V_\lambda[g]$ to conclude
  that $\delta$ is exacting in $W$.
 
  Assume, towards a contradiction, that $\delta$ is not an exacting cardinal in
  $V_\lambda[G]$. By Lemma \ref{fact:CharExacting}, this means that, in
  $V_\lambda[G]$, there exists a non-empty ordinal definable subset $B$ of
  $V_{\delta+1}$ with the property that for all $x,y\in B$, there is no
  non-trivial elementary embedding of $\langle V_\delta,\in,x\rangle$ into
  $\langle V_\delta,\in,y\rangle$.
  Since $\PPP_{D\setminus(\delta+1)}=\PPP_{D\setminus\delta_*}$ and
  $G_{{<}\delta}\times g$ is $\PPP_{D\cap\delta_*}$-generic over $V$,
  Lemma~\ref{MagProdProps}.\eqref{MagProdProp23} applied at $\delta_*$ shows
  that forcing with $\PPP_{D\setminus(\delta+1)}$ over $W$ does not add new
  elements of $V_{\delta+\omega}$. In particular, we have
  $V_{\delta+1}^{V_\lambda[G]}=V_{\delta+1}^{W}$ and $B\in W$.
  Since $V_\lambda$ is a model of "$V=\gHOD$", we now know that
  $\PPP_{D\setminus(\delta+1)}$ is an element of $\HOD^{W}$. Moreover, the fact
  that Lemma~\ref{MagProdProps}.\eqref{MagProdProp22} ensures that
  $\PPP_{D\setminus(\delta+1)}$ is weakly homogeneous in $W$ allows us to
  conclude that $\OD^{V_\lambda[G]}\cap W$ is contained in $\OD^{W}$ and hence
  that $B$ is a non-empty ordinal definable subset of $V_{\delta+1}$ in $W$.
  Since $\delta$ is an exacting cardinal in $W$, we can use
   Lemma \ref{fact:CharExacting} to find $x,y\in B$ and a non-trivial elementary
  embedding of $\langle V_\delta,\in,x\rangle$ into
  $\langle V_\delta,\in,y\rangle$ in $W$. But, since $W\subseteq V_\lambda[G]$,
  this embedding is also a non-trivial elementary embedding of
  $\langle V_\delta,\in,x\rangle$ into $\langle V_\delta,\in,y\rangle$ in
  $V_\lambda[G]$, contradicting our choice of $B$.
\end{proof}

We are now ready to show that the existence of an I2-embedding implies the existence of a transitive model of $\ZFC$ with a proper class of exacting cardinals:

\begin{proof}[Proof of Theorem \ref{thm: a proper class of exactings}]
  Assume that $V$ is a model of $\ZFC$ containing an I2-embedding.
  Then there exists a class forcing extension $N$ of $V$ which is a model of
  $$\ZFC ~ + ~ \textit{``}V=\gHOD\textit{''} ~ + ~
  \textit{``There is an I2-embedding''}$$
  and in which, moreover, $V_\lambda$ is a model of $V=\gHOD$ for the first
  non-trivial fixed point $\lambda$ of that embedding.
  Work in $N$, let $\kappa<\lambda$, $D$ and $\PPP_D$ be as above and let $G$
  be $\PPP_D$-generic over $N$.
  By Lemma \ref{lemma:ManyExacting}, every $\delta\in D$ is an exacting
  cardinal in $N[G]_\lambda$, and $D$ is unbounded in $\kappa$. Since
  $\kappa=\sup(D)$ is inaccessible in $N$, Corollary
  \ref{corollary:MagidorSupportPreserveInaccessible} shows that $\kappa$ is
  inaccessible in $N[G]$ and therefore $N[G]_\kappa$ is a model of $\ZFC$. By
  Proposition~\ref{prop:ExactingRankSegments}, every $\delta\in D$ is then an
  exacting cardinal in $N[G]_\kappa$.
  Consequently, $N[G]_\kappa$ is a model of $\ZFC$ in which unboundedly many
  ordinals are exacting cardinals, {i.e.,} a set-sized model of the theory
  $$\ZFC ~ + ~ \textit{``There is a proper class of exacting cardinals''}.$$
  By the L\"owenheim--Skolem theorem and Mostowski collapsing, it
  follows that $N[G]$ contains a countable transitive model of this theory.
  Since the existence of such a model can be expressed by a
  $\Sigma^1_2$-sentence, \emph{Shoenfield absoluteness} ensures that such a
  model exists in $V$.
\end{proof}


\subsection{Limits of exacting cardinals}\label{sec:ExactingLimits}
In the model constructed in the above proof, no exacting cardinal is a limit of exacting cardinals. Theorem \ref{theorem:ExactingLimit} shows that there might be no way around this limitation. We now present the short proof of this result. It should be noted that this argument also works for limits of exacting cardinals with weaker large cardinal properties.

\begin{proof}[Proof of Theorem \ref{theorem:ExactingLimit}]
   \eqref{item:ExactingLimit1} Assume that $\eta$ is an exacting limit of exacting cardinals. Since $\HOD^{V_\eta}$ is a subset of $\HOD$ and exacting cardinals are regular in $\HOD$, it follows that, in $V_\eta$, there is a proper class of
    singular cardinals that are regular in $\HOD$. Since $\eta$ is exacting,  there is an I3-embedding $\map{i}{V_\eta}{V_\eta}$ and therefore, in    $V_\eta$, there is an extendible cardinal below a huge cardinal. By
    Woodin's HOD Dichotomy (Theorem \ref{thm:HODdichotomy}), this contradicts the Weak HOD Conjecture. 

   \eqref{item:ExactingLimit2} Assume, towards a contradiction, that there is a limit cardinal $\kappa$ of uncountable cofinality with the property that the set $E$ of exacting
    cardinals below $\kappa$ is stationary in $\kappa$. Then $E$ intersects    the set of all limit points of $E$ in $\kappa$ and hence there exists an exacting limit of exacting cardinals. This contradicts \eqref{item:ExactingLimit1}. 
\end{proof}

In the next section, we show that the consistency of $\ZF$ with sufficiently strong \emph{large cardinals beyond choice} ({i.e.,} in a scenario where the \emph{HOD Conjecture} is false, see \cite{BagKoelWoo}) implies the consistency of $\ZFC$ with the statements listed in Theorem \ref{theorem:ExactingLimit}.


\section{Exacting cardinals from large cardinals beyond choice}\label{sec: exacting from beyond choice}

Throughout this section, we work with ground models satisfying $\ZF$ and containing large cardinals that cause the Axiom of Choice to fail.


\subsection{Ultraexacting limits of ultraexacting cardinals}
Given a supercompact cardinal $\delta$ (see Definition \ref{def:ZFsupercompact}),  we let $\QQQ_\delta$ denote the partial order constructed by Woodin in the proof of {\cite[Theorem 226]{SEM1}}. Key properties of this partial order that were proven in \cite{SEM1} will be listed in the proof of the next lemma that is a variations of {\cite[Theorem 226]{SEM1}} and {\cite[Lemma 6.7]{ABL}}. Remember that, in the absence of the Axiom of Choice, we define ultraexactingness as in Definition \ref{definition:ZFultraexacting}.

\begin{lemma}[ZF]\label{lemma:WoodinForcingExacting}
  If $\lambda$ is a rank-Berkeley cardinal and $\delta>\lambda$ is a supercompact cardinal with $V_\lambda\prec {V_\delta}$, then ${\one}\Vdash_{\QQQ_\delta}\textit{$``\lambda$ is an ultraexacting cardinal }"$. 
\end{lemma}

\begin{proof}
  Assume, towards a contradiction, that the statement of the lemma fails.     We may then assume that $\delta$ is the minimal supercompact cardinal above $\lambda$ satisfying $V_\lambda\prec {V_\delta}$ for which the given conclusion fails. Since the partial order $\QQQ_\delta$ is weakly homogeneous and definable in $V_\delta$ by a formula without parameters, our assumption ensures that the cardinal $\delta$ is definable by a formula with parameter $\lambda$. The analysis of \cite{SEM1} now shows that $\QQQ_\delta$ is forcing equivalent to a two-step iteration $\QQQ_\lambda*\dot{\PPP}$ satisfying the following: 
    \begin{enumerate}
        \item The partial order $\QQQ_\lambda$ is definable in $V_{\lambda+1}$ by a formula without parameters and forcing with $\QQQ_\lambda$ preserves unboundedly many regular cardinals below $\lambda$. 

        \item\label{item:WoodinForcing2} Given a non-trivial elementary embedding $\map{j}{V_{\lambda+1}}{V_{\lambda+1}}$ with first non-trivial fixed point $\lambda$,   there exists a condition $p_j$ in $\QQQ_\lambda$ with the property that $$G ~ = ~ \Set{p\in\QQQ_\lambda}{j(p)\in G}$$ holds whenever $G$ is $\QQQ_\lambda$-generic over $V$ with $p_j\in G$. 

        \item\label{item:WoodinForcing3} If $G$ is $\QQQ_\lambda$-generic over $V$, then, in $V[G]$, the partial order $\dot{\PPP}^G$ is ${<}\lambda^+$-closed, weakly homogeneous and definable in $V_\delta$ without parameters. 

        \item\label{item:WoodinForcing4} If $G$ is $\QQQ_\lambda$-generic over $V$, then $\dc_\lambda$ holds in $V[G]$. 
    \end{enumerate}

  Pick $\zeta>\delta$ such that $V_\zeta$ is sufficiently elementary in $V$. Then there exists an elementary embedding $\map{j}{V_\zeta}{V_\zeta}$ such that $\crit(j)<\lambda$ and $\lambda$ is the first non-trivial fixed point of $j$. We then have $\cof{\lambda}=\omega$,  $j(\delta)=\delta$ and $j(\QQQ_\lambda)=\QQQ_\lambda$. Let $p_j$ be the condition given by \eqref{item:WoodinForcing2} and let $G*H$ be $(\QQQ_\lambda*\dot{\PPP})$-generic over $V$ with $p_j\in G$. Then $V_\zeta[G]=V[G]_\zeta$, $\zeta\in(C^{(2)})^{V[G]}$ and $j[G]\subseteq G$. In particular, there is a canonical lifting $$\map{j_G}{V[G]_\zeta}{V[G]_\zeta}$$ of $j$ to $V[G]_\zeta$. 
  We can now apply Theorem \ref{theorem:UltraexactingInnerModels} in $V[G]$ with $N=V[G]$ and $X_0=X_1=V[G]_\zeta$ to conclude that $\lambda$ is ultraexacting in $V[G]$. 

  Now, by \eqref{item:WoodinForcing3} and \eqref{item:WoodinForcing4}, forcing with $\dot{\PPP}^G$ over $V[G]$ does not add new subsets of $V_\lambda$. Moreover, \eqref{item:WoodinForcing3} ensures that the partial order $\dot{\PPP}^G$ is weakly homogeneous and ordinal definable in $V[G]$. This allows us to apply Lemma \ref{prop:PreserveExactingGoodForcing}.\eqref{prop:ZFforcingPreserve2} to conclude that $\lambda$ is an ultraexacting cardinal in $V[G,H]$. 

  These computations contradict our initial assumption,  because, in $V$,  the partial order $\QQQ_\delta$ is weakly homogeneous and forcing equivalent to $\QQQ_\lambda*\dot{\PPP}$. 
\end{proof}

\begin{cor}[$\mathrm{ZF}^2$]
    If there exists a super Reinhardt cardinal, then there exists a set-sized transitive model of $\ZFC$  with  a proper class of ultraexacting cardinals that are limits of ultraexacting cardinals. 
\end{cor}

\begin{proof}
  Let $\delta$ be a super Reinhardt cardinal. Then $\delta$ is supercompact and therefore a combination of results in \cite{SEM1} with Proposition \ref{prop:ZFExactingDown}, Lemma \ref{lemma:SuperReinhardtExacting} and Lemma \ref{lemma:WoodinForcingExacting} shows that if $G$ is $\QQQ_\delta$-generic over $V$, then $V[G]_\delta$ is a model of $\ZFC$  with a proper class of ultraexacting cardinals that are limits of ultraexacting cardinals. In particular, there is a generic extension of the ground model $V$ that contains a countable transitive model of $\mathrm{ZFC}$ with a proper class of ultraexacting cardinals that are limits of ultraexacting cardinals. Shoenfield absoluteness then ensures that such a model already exists in $V$. 
\end{proof}


\subsection{Stationary limits of ultraexacting cardinals}
Given the above, it is natural to ask whether large cardinals beyond choice can provide richer patterns of exacting and ultraexacting cardinals in the universe of sets. Specifically, in the light of Theorem \ref{theorem:ExactingLimit}.\eqref{item:ExactingLimit2}, one might ask whether these axioms can be used to construct models of $\ZFC$ with stationary limits of exacting cardinals. We address this issue in this section.

\begin{lemma}[$\ZF$]\label{lemma:StatUltra}
    If $\delta$ is a cardinal that is both supercompact and totally Reinhardt, $G$ is $\QQQ_\delta$-generic over $V$ and $H$ is $\Col({<}\delta,V_\delta)$-generic over $V[G]$, then, in $V[G,H]$,  every closed unbounded subset of $\delta$ contains an ultraexacting cardinal. 
\end{lemma}

\begin{proof}
    The results of \cite{SEM1} show that, in $V$, the partial order $\QQQ_\delta$ is forcing equivalent to a two-step iteration $\QQQ*\dot{\QQQ}$ such that $\QQQ$ is a partial order in $V_\delta$ that forces $\dc$ to hold and $\dot{\QQQ}$ to be a $\sigma$-closed partial order that  is a subset of $V_\delta$. Let $G_0*G_1$ be $(\QQQ*\dot{\QQQ})$-generic over $V$ with $V[G]=V[G_0,G_1]$. We then know that $\delta$ is supercompact and totally Reinhardt in $V[G_0]$. 
    Moreover, the results of \cite{SEM1}  show that every element of $V[G]_\delta$ has a $\dot{\QQQ}^{G_0}$-name in $V[G_0]_\delta$. 
    Therefore, it follows that $V[G,H]$ is a forcing extension of $V[G_0]$ using a partial order $\CCC$ that is a subset of $V[G_0]_\delta$ and $\sigma$-closed in $V[G_0]$. 

    Fix a $\CCC$-name $\dot{C}\in V[G_0]$ for a closed unbounded subset of $\delta$ and a condition $p_*$ in $\CCC$. We can now find  $A\in V[G_0]_{\delta+1}$ coding the condition $p_*$, the ordering $\leq_\CCC$ of $\CCC$ and the subset $$D ~ = ~ \Set{\langle p,\beta\rangle\in\CCC\times\delta}{p\Vdash^{V[G_0]}_\CCC"\beta\in\dot{C}\hspace{0.8pt}"}$$ of $V_\delta$. Since $\delta$ is totally Reinhardt in $V[G_0]$, we can find $\kappa<\delta$ that is $A$-super Reinhardt in $\langle V[G_0]_\delta,V[G_0]_{\delta+1},\in\rangle$. Our setup then ensures that the condition $p_*$ is an element of $\CCC\cap V[G_0]_\kappa$.

    Given $p\in\CCC\cap V[G_0]_\kappa$ and $\alpha<\kappa$, there exists $\gamma_{\alpha,p}<\delta$ with the property that there exist $\alpha<\beta<\gamma_{\alpha,p}$ and  $q\in\CCC\cap V[G_0]_{\gamma_{\alpha,p}}$ with $q\leq_\CCC p$ and $\langle q,\beta\rangle\in D$. Since $\delta$ is strongly inaccessible in $V[G_0]$, we can find $\rho<\delta$ with $\gamma_{\alpha,p}<\rho$ for all $p\in\CCC\cap V[G_0]_\kappa$ and $\alpha<\kappa$. 
    Therefore, we can find a non-trivial elementary embedding $\map{j}{V[G_0]_\delta}{V[G_0]_\delta}$  with $\crit(j)=\kappa$, $j(\kappa)>\rho$ as well as $$j({\leq_\CCC}\cap V_\alpha) ~ = ~ {\leq_\CCC} \cap V_{j(\alpha)}$$ and  $$j(D  \cap  V_\alpha) ~ = ~ D  \cap  V_{j(\alpha)}$$ for all $\alpha<\kappa$. 
    The elementarity of $j$ then ensures that for every condition  $p\in\CCC\cap V[G_0]_\kappa$ and every $\alpha<\kappa$, there exists a condition  $q\in\CCC\cap V[G_0]_\kappa$ and $\alpha\leq\beta<\kappa$ with $q\leq_\CCC p$ and $\langle q,\beta\rangle\in D$. 
    
    Let $\seq{\kappa_n}{n<\omega}$ denote the critical sequence of $j$ and set $\lambda=\sup_{n<\omega}\kappa_n$. Since $\delta$ is a regular cardinal in $V[G_0]$, we have $\lambda<\delta$.  
    Elementarity then ensures that for every $n<\omega$, every $p\in\CCC\cap V[G_0]_{\kappa_n}$ and every $\alpha<\kappa_n$, there exist $q\in\CCC\cap V[G_0]_{\kappa_n}$ and $\alpha\leq\beta<\kappa_n$ with the property that  $q\leq_\CCC p$ and $\langle q,\beta\rangle\in D$. This allows us to use $\dc$ in $V[G_0]$ to find a $\leq_\CCC$-descending sequence $\seq{p_n\in\CCC\cap V[G_0]_{\kappa_n}}{n<\omega}$ with $p_0=p_*$ and the property that for all $n<\omega$, there exists $\kappa_n\leq\beta_n<\kappa_{n+1}$ with $\langle q_{n+1},\beta_n\rangle\in D$. Since $\CCC$ is $\sigma$-closed in $V[G_0]$, we can find a condition $p_\omega$ in $\CCC$ with $p_\omega\leq_\CCC p_n$ for all $n<\omega$. 
    
    Let $F$ be $\CCC$-generic over $V[G_0]$ with $p_\omega\in F$. Then $\dot{C}^F\cap[\kappa_n,\kappa_{n+1})\neq\emptyset$ holds for all $n<\omega$ and therefore we know that $\lambda$ is an element of $\dot{C}^F$. 
    Earlier arguments show that $\lambda$ is a rank-Berkeley cardinal in $V[G_0]_\delta$ and $V[G_0]_\lambda\prec V[G_0]_\delta$ holds. By Proposition \ref{prop:LocalRankBerkeleySupercompact}, this means that $\lambda$ is a rank-Berkeley cardinal in $V[G_0]$. Moreover, the analysis in \cite{SEM1} shows that $\dot{\QQQ}^{G_0}=\QQQ_\delta^{V[G_0]}$ holds, {i.e.,} $\dot{\QQQ}^{G_0}$ satisfies the definition of Woodin's partial order in $V[G_0]$. 
    Let $F_0$ denote the filter on $\dot{\QQQ}^{G_0}$ induced by $F$. Then Lemma \ref{lemma:WoodinForcingExacting} shows that $\lambda$ is an ultraexacting cardinal in $V[G_0,F_0]$. Next, Proposition \ref{prop:PreserveExactingGoodForcing} ensures that $\lambda$ is ultraexacting in $V[G_0,F]$. 

    The above computations show that below every condition $p_*$ in $\CCC$, there exists a condition that forces $\dot{C}$ to contain an ultraexacting cardinal. A density argument now yields the statement of the lemma. 
\end{proof}

Putting together the previous results, we obtain the following theorem  that directly implies the statement of Theorem \ref{thm:SimplifiedConsStatLimitUltra}:

\begin{theorem}[$\ZF$]
        If $\delta$ is a cardinal that is both supercompact and totally Reinhardt, $G$ is $\QQQ_\delta$-generic over $V$ and $H$ is $\Col({<}\delta,V_\delta)$-generic over $V[G]$, then $L(V[G]_\delta)[H]$ is a model of $\ZFC$ in which $\delta$ is a regular cardinal that is a stationary limit of ultraexacting cardinals. 
\end{theorem}

\begin{proof}
    Since $\dc_{{<}\delta}$ holds in $V[G]$ and the partial order $\Col({<}\delta,V_\delta)$ is ${<}\delta$-closed in $V[G]$, we have    $V[G]_\delta=V[G,H]_\delta$ and the fact that $H$ codes a well-ordering of $V[G]_\delta$ ensures that $L(V[G]_\delta)[H]$ is a model of $\ZFC$. 
    Let $C$ be a closed unbounded subset of $\delta$ in $L(V[G]_\delta)[H]$. Then Lemma \ref{lemma:StatUltra} shows that $C$ contains an ordinal $\lambda$ that is an ultraexacting cardinal in $V[G,H]$. In this situation, the fact that $V[G]_\delta=V[G,H]_\delta$ allows us to use Proposition \ref{prop:ExactingDownToL} to see that $\lambda$ is an ultraexacting cardinal in $L(V[G]_\delta)$. 
    Finally, since $V[G]_\delta=V[G,H]_\delta$,     we can apply Proposition \ref{prop:PreserveExactingGoodForcing} in $L(V[G]_\delta)$ to conclude that $\lambda$ is also an ultraexacting cardinal in $L(V[G]_\delta)[H]$. 
\end{proof}


\section{The Levy--Solovay phenomenon}\label{sec: smallforcing}

In this section, we investigate the Levy--Solovay phenomenon in the context of exacting and ultraexacting cardinals by proving Theorem \ref{theorem:LevySolovayBeyondHOD} and deriving consequences of the combination of this result with Lemma \ref{lemma:ExactingSmallForcing}.  %
 The below proof heavily relies on Woodin's  stationary tower forcing (see \cite{stationarytower}) and {\cite[Lemma 147]{SEM1}} that allows us to compute the restrictions of lifted elementary embeddings to generic ultrapowers of the ground model.

\begin{proof}[Proof of Theorem \ref{theorem:LevySolovayBeyondHOD}]
  In the following, assume that $\lambda$ is a limit of Woodin cardinals and $\PPP\in V_\lambda$ is a partial order that forces $\lambda$ to have countable cofinality. Fix a Woodin cardinal $\delta<\lambda$ with $\PPP\in V_\delta$ and  an inaccessible cardinal $\delta<\rho<\lambda$. 
    Let $\QQQ_{{<}\delta}$ denote Woodin's countable stationary tower forcing (see \cite{stationarytower} for details). 

   Now, let $g$ be $\QQQ_{<\delta}$-generic over $V$ and let    $$\map{j_g}{V}{M_g\s V[g]}$$ denote the generic elementary embedding induced by $g$. Standard stationary tower forcing arguments (see \cite{stationarytower}) then show that the following statements hold: 
   \begin{itemize}
       \item $\crit(j_g)=\omega_1^V$ and $j_g(\omega_1^V)=\delta$. 
       
       \item $M_g = \Set{j_g(f)(\alpha)}{\map{f}{\omega_1^V}{V}, ~ \alpha <\delta}$. 
       
       \item  $({}^\omega M_g)^{V[g]}\s M_g$.
   \end{itemize}

   Since $\mathcal{P}(\PPP)^V$ is countable in $V[g]$, we can find a filter $G$ on $\PPP$ in $V[g]$ that is generic over $V$. 
  %
    %
    In addition, results of Hamkins, Laver and Woodin (see \cite{MR2364192})  show that there is a parameter $p\in V_\rho$ with the property that the class $V$ is definable in $V[g]$ by a formula with parameter $p$. This directly implies that the class $M_g$ is definable in $V[g]$ from the parameters $g$ and $p$. 
    Finally, the construction of $M_g$ as an ultrapower ensures that all sufficiently large inaccessible cardinals below $\lambda$ are fixed points of $j_g$. Since the embedding $j_g$ is continuous at ordinal of countable cofinality in $V$, this directly yields the following statements:

   \begin{claim*}
    \begin{enumerate}
        \item If $\lambda$ has countable cofinality in $V$, then $j_g(\lambda)=\lambda$. 

        \item If $\lambda$ has uncountable cofinality in $V$ and $C$ is a closed unbounded subset of $\lambda$ in $V$, then $\lambda\in j_g(C)$. \qed 
    \end{enumerate}
   \end{claim*}


   \begin{claim*}
     If $\map{j}{V[g]_\lambda}{V[G]_\lambda}$ is an I3-embedding with $\crit(j)>\rho$ in $V[g]$, then the map $j\restriction(M_g)_\lambda$ is an element of $M_g$. 
   \end{claim*}

   \begin{proof}[Proof of the claim]
     Let $\seq{\kappa_n}{n<\omega}$ denote the critical sequence of $j$. 
     
     Fix some $0<n<\omega$. Since $\QQQ_{{<}\delta}$ is an element of $V_{\kappa_0}$, we can apply {\cite[Lemma 4]{MR2364192}} to see that $j\restriction V_{\kappa_n}$ is an element of the ground model $V$ and we let $E$ denote the $(\kappa_0,\kappa_n)$-extender derived from this embedding in $V$. Then $E$ is $\rho$-complete in $V$ (as defined in {\cite[Definition 146]{SEM1}}). 
     Let $\map{i_E}{V}{M_E}$ denote the  ultrapower embedding given by $E$ in $V$. The definition of $E$ then ensures that $i_E\restriction V_{\kappa_{n-1}}=j\restriction V_{\kappa_{n-1}}$. Let $E^\prime$ denote the canonical $(\kappa_0,\kappa_n)$-extender induced by $E$ in $V[g]$ and let $\map{i_{E^\prime}}{V[g]}{M_{E^\prime}}$ denote the corresponding ultrapower embedding constructed in $V[g]$. Then $i_{E^\prime}$ is the canonical lifting of $i_E$ to the $\QQQ_{{<}\delta}$-generic extension and this directly implies that 
       \begin{equation}\label{equation:ExactingEmbeddingsCoincides}
           i_{E^\prime}\restriction V[g]_{\kappa_{n-1}} ~ = ~ j\restriction V[g]_{\kappa_{n-1}},
       \end{equation} because for each $z\in V[g]_{\kappa_{n-1}}$, there is a $\QQQ_{{<}\delta}$-name $\tau\in V_{\kappa_{n-1}}$ with $\tau^g=x$ and $$j(x) ~ = ~ j(\tau^g) ~ = ~ j(\tau)^g ~ = ~ i_E(\tau)^g ~ = ~ i_{E^\prime}(x).$$

     Set $F=j_g(E)$. Note that our setup ensures that $j_g(\kappa_\ell)=\kappa_\ell$ holds for all $\ell<\omega$ and therefore we know that $F$ is a $(\kappa_0,\kappa_n)$-extender in $M_g$. Let $\map{i_F}{M_g}{M_F}$ denote the corresponding ultrapower embedding constructed in $M_g$. Earlier observations now show that the extender $E$ and the embedding $j_g$ satisfy the assumptions of {\cite[Lemma 147]{SEM1}} and this result then shows that $i_{E^\prime}\restriction M_g  =  i_F$ holds. In combination with \eqref{equation:ExactingEmbeddingsCoincides}, this shows that $$j\restriction(M_g)_{\kappa_{n-1}} ~ = ~ i_F\restriction(M_g)_{\kappa_{n-1}} ~ \in ~ M_g.$$

    The above computations show that $j\restriction(M_g)_{\kappa_n}$ is an element of $M_g$ for every $n<\omega$. Since $M_g$ is closed under countable sequences in $V[g]$, we can now conclude that $j\restriction(M_g)_\lambda$ is also an element of $M_g$.  
   \end{proof}

   \begin{claim*}
       If $\lambda$ is an exacting cardinal in $V[G]$, then $\lambda$ is an exacting cardinal in $M_g$. 
   \end{claim*}

   \begin{proof}[Proof of the claim]
       Assume, towards a contradiction, that $\lambda$ is not an exacting cardinal in $M_g$. By Theorem \ref{fact:CharExacting}, this shows that, in $M_g$, there is a non-empty ordinal definable subset $A$ of $V_{\lambda+1}$ with the property that for all $x,y\in A$, there is no non-trivial elementary embedding of $\langle V_\lambda,\in,x\rangle$ into $\langle V_\lambda,\in,y\rangle$. Let $A$ be the least such subset in $M_g$ with respect to the canonical well-ordering of ordinal definable sets in $M_g$. Then the set $\{A\}$ is definable in $M_g$ by a formula with parameter $\lambda$. In this situation, our earlier observations show that the set $\{A\}$ is definable in $V[g]$ by a formula with parameters $\lambda$, $g$ and $p$.

       Since $V[g]$ is a forcing extension of $V[G]$ using a partial order of size less than $\rho$, our assumption allows us to use Lemma \ref{lemma:ExactingSmallForcing}.\eqref{item:ExactingUltraSmallForcing1} to show that $\lambda$ is an exacting cardinal in $V[g]$. 
       Pick   $\zeta>\lambda$ such that $V[g]_\zeta$ is sufficiently elementary in $V[g]$ and use the exactingness of $\lambda$ in $V[g]$ to  find  an elementary submodel $X$ of $V[g]_\zeta$ in $V[g]$ containing $V[g]_\lambda\cup\{\lambda\}$  as well as   an elementary embedding $\map{j}{X}{V[g]_\zeta}$ in $V[g]$ satisfying $j\restriction\rho=\id_\rho$, $j\restriction\lambda\neq\id_\lambda$ and $j(\lambda)=\lambda$. 
       The fact that $\lambda$, $g$ and $p$ are elements of $X$ that are fixed by $j$ then ensures that $A$ and $(M_g)_\lambda$ are  elements of $X$ that are fixed by $j$. 
       By elementarity, there exists an $x\in A\cap X$. Set $y=j(x)\in A$. Then $j\restriction(M_g)_\lambda$ is a non-trivial  elementary embedding of $\langle(M_g)_\lambda,\in,x\rangle$ into $\langle(M_g)_\lambda,\in,y\rangle$. 
       Since $j\restriction V[g]_\lambda$ is an I3-embedding with critical point greater than $\rho$ in $V[g]$, our previous claim shows that the map $j\restriction(M_g)_\lambda$ is an element of $M_g$, contradicting our choice of $A$. 
   \end{proof}

      \begin{claim*}
       If $\lambda$ is an ultraexacting cardinal in $V[G]$, then $\lambda$ is an ultraexacting cardinal in $M_g$. 
   \end{claim*}

   \begin{proof}[Proof of the claim]
       Assume, towards a contradiction, that $\lambda$ is not an ultraexacting cardinal in $M_g$.  Theorem \ref{fact:CharUltraexacting} then implies that, in $M_g$, there is an ordinal definable subset $A$ of $V_{\lambda+1}$ with the property that there is no non-trivial elementary embedding of $\langle V_{\lambda+1},\in,A\rangle$ into itself. Let $A$ be the least such subset in $M_g$ with respect to the canonical well-ordering of ordinal definable sets in $M_g$. Since the set $\{A\}$ is again definable in $M_g$ by a formula with parameter $\lambda$, it follows that the set $\{A\}$ is definable in $V[g]$ by a formula with parameters $\lambda$, $g$ and $p$. 

      Next, we apply Lemma \ref{lemma:ExactingSmallForcing}.\eqref{item:ExactingUltraSmallForcing2} to see that $\lambda$ is an ultraexacting cardinal in $V[g]$. 
       Fix  $\zeta>\lambda$ with the property that $V[g]$ is sufficiently elementary in $V[g]$,  an elementary submodel $X$ of $V[g]_\zeta$ in $V[g]$ containing $V[g]_\lambda$  and an elementary embedding $\map{j}{X}{V[g]_\zeta}$ in $V[g]$ satisfying $j\restriction\rho=\id_\rho$, $j\restriction\lambda\neq\id_\lambda$, $j(\lambda)=\lambda$ and $j\restriction V[g]_\lambda\in X$. 
       Then $j\restriction V[g]_\lambda$ is an I3-embedding in $V[g]$ and an earlier claim shows that $j\restriction(M_g)_\lambda$ is an element of $M_g$. Define $$\map{j_* ~ = ~ (j\restriction V[g]_\lambda)_+^{V[g]}}{V[g]_{\lambda+1}}{V[g]_{\lambda+1}}$$ (see \S\ref{subsection:RankIntoRank}). An application of  {\cite[Lemma 3.7]{ABL}} then shows that $j_*$ is an I1-embedding that is an element of $X$ and satisfies  $$j\restriction(V[g]_{\lambda+1}\cap X) ~ = ~ j_*\restriction(V[g]_{\lambda+1}\cap X).$$
       Since our setup ensures that $A$ and $(M_g)_{\lambda+1}$ are elements of $X$ that are fixed by $j$, it follows that, in $X$, the map $j_*\restriction(M_g)_{\lambda+1}$ is an elementary embedding of $\langle(M_g)_{\lambda+1},\in,A\rangle$ into itself. The correctness properties of $X$ then ensure that $j_*\restriction(M_g)_{\lambda+1}$ has the same property in $V[g]$.

       Finally, the fact that $j\restriction(M_g)_\lambda$ is an I3-embedding in $M_g$, allows us to define $$\map{i ~ = ~ (j\restriction(M_g)_\lambda)_+^{M_g}}{(M_g)_{\lambda+1}}{(M_g)_{\lambda+1}}.$$ 
       Since $i\restriction(M_g)_\lambda=j\restriction(M_g)_\lambda=j_*\restriction(M_g)_\lambda$,  it follows that $$j_*\restriction(M_g)_{\lambda+1} ~ = ~ i ~ \in ~ M_g$$ and therefore we can conclude that $M_g$ contains a non-trivial elementary embedding of $\langle(M_g)_{\lambda+1},\in,A\rangle$ into itself, contradicting our choice of $A$. 
   \end{proof}

  We are now ready to prove the two parts of the theorem:

  \eqref{item:LevySolovayBeyondHOD1} Assume that $\lambda$ is a cardinal with the property that $$\one\Vdash_\PPP``\textit{$\lambda$ is an exacting cardinal }"$$ holds for some partial order $\PPP\in V_\lambda$. 
  Since exacting cardinals are limits of Woodin cardinals,  {\cite[Corollary 25]{MR2063629}} shows that $\lambda$ is a limit of Woodin cardinals in $V$. In particular, our assumption allows us to  carry out the above constructions and conclude that $\lambda$ is an exacting cardinal in the model $M_g$. 
  If $\lambda$ has countable cofinality in $V$, then an earlier claim shows that $j_g(\lambda)=\lambda$ and the elementarity of $j_g$ implies that $\lambda$ is an exacting cardinal in $V$. Moreover, if $\lambda$ has uncountable cofinality and $C$ is a closed unbounded subset of $\lambda$ in $V$, then an earlier claim shows that $\lambda$ is an element of $j_g(C)$ and we can conclude that $C$ contains an exacting cardinal in $V$.

  \eqref{item:LevySolovayBeyondHOD2} Assume that $\lambda$ is a cardinal with the property that $$\one\Vdash_\PPP``\textit{$\lambda$ is an ultraexacting cardinal }"$$ holds for some partial order $\PPP\in V_\lambda$. 
    As above, this assumption allows us to carry out the above constructions and conclude that $\lambda$ is an ultraexacting cardinal in the model $M_g$. In addition, we can again conlcude that, if $\lambda$ has countable cofinality in $V$, then $\lambda$ is an ultraexacting cardinal in $V$, and, if $\lambda$ has uncountable cofinality in $V$, then every closed unbounded subset of $\lambda$ in $V$ contains a cardinal that is ultraexacting in $V$.  
\end{proof}

By combining Theorem \ref{theorem:LevySolovayBeyondHOD} with Lemma \ref{lemma:ExactingSmallForcing}, we immediately obtain the following Levy--Solovay result that is restricted to countable cofinalities:

\begin{cor}\label{cor:LevySolovayCountableCof}
 Let $\lambda$ be a cardinal of countable cofinality. 
 \begin{enumerate}
     \item The following statements are equivalent: 
        \begin{enumerate}
      \item $\lambda$ is an exacting cardinal. 

      \item ${\one}\Vdash_\PPP\textit{$``\lambda$ is an exacting cardinal }"$ holds for every partial order $\PPP\in H(\lambda)$.  
      
      \item ${\one}\Vdash_\PPP\textit{$``\lambda$ is an exacting cardinal }"$ holds for some partial order $\PPP\in H(\lambda)$.
  \end{enumerate}

       \item The following statements are equivalent: 
        \begin{enumerate}
      \item $\lambda$ is an ultraexacting cardinal. 

      \item ${\one}\Vdash_\PPP\textit{$``\lambda$ is an ultraexacting cardinal }"$ holds for every partial order $\PPP\in H(\lambda)$.  
      
      \item ${\one}\Vdash_\PPP\textit{$``\lambda$ is an ultraexacting cardinal }"$ holds for some partial order $\PPP\in H(\lambda)$. \qed
  \end{enumerate}
 \end{enumerate} 
\end{cor}

Next, we now prove that the weak HOD Conjecture yields  Levy-Solovay theorems for exacting and ultraexacting cardinals:

\begin{proof}[Proof of Theorem \ref{theorem: WeakHODConjecture and LevySolovay}]
  Assume that the weak HOD Conjecture is true. By Theorem \ref{theorem:ExactingLimit}, we then know that exacting cardinals are not limits of exacting cardinals. An application of Theorem \ref{theorem:LevySolovayBeyondHOD} then shows that no cardinal $\lambda$ of uncountable cofinality is forced to be exacting by a partial order in $H(\lambda)$. In particular, since exacting cardinals have countable cofinality and ultraexacting cardinals are exacting, the equivalences stated in the theorem directly follow from the ones provided by Corollary \ref{cor:LevySolovayCountableCof}.  
\end{proof}

We end this section by stating another  consequence of Theorem \ref{theorem:LevySolovayBeyondHOD}:

\begin{cor}\label{cor:LevySolovayAbove}
 Let $\kappa$ be a cardinal. 
 \begin{enumerate}
     \item The following statements are equivalent: 
      \begin{enumerate}
          \item There is an exacting cardinal greater than $\kappa$. 

          \item There is a partial order of cardinality at most $\kappa$ that forces a cardinal greater than $\kappa$ to be exacting. 
      \end{enumerate}

     \item The following statements are equivalent: 
      \begin{enumerate}
          \item There is an ultraexacting cardinal greater than $\kappa$. 

          \item There is a partial order of cardinality at most $\kappa$ that forces a cardinal greater than $\kappa$ to be ultraexacting. \qed
      \end{enumerate} 
 \end{enumerate}
\end{cor}


\section{Exacting cardinals and  HOD}\label{sec:ExactingInHOD}

 In this section, we analyze the extent of large cardinal properties that exacting cardinals provably possess in $\HOD$. This is a line of research initiated by Cheng--Hamkins--Friedman in \cite{ChengFriedmanHamkins} and continued by Apter--Friedman--Fuchs \cite{ApterFriedmanFuchs} and Goldberg--Osinski--Poveda \cite{GOP}.
 
 The section is divided in  \S\S\ref{sec: exactings in HOD} and \S\S\ref{sec: successors of exactings}. In \S\S\ref{sec: exactings in HOD} we show that exacting cardinals $\lambda$ are always subtle in HOD and show that if $V_\lambda\models ``\HOD$ hypothesis'' then $\lambda$ possesses strong large cardinal properties in $\HOD$. In  \S\S\ref{sec: successors of exactings} we prove (assuming appropriate large cardinals) the consistency with $\mathrm{ZF}$ of $``\lambda$ is exacting and $\lambda^+$ is extendible in $\HOD_x$ for all $x\s \lambda$".


\subsection{Exacting cardinals in HOD}\label{sec: exactings in HOD}
Recall that an infinite cardinal $\kappa$ is \emph{subtle} (see \cite{jensennotes}) if for every sequence $\seq{E_\alpha}{\alpha<\kappa}$ with $E_\alpha\subseteq\alpha$ for all $\alpha<\kappa$ and every closed unbounded subset $C$ of $\kappa$, there exist $\alpha,\beta\in C$ with $\alpha<\beta$ and $E_\alpha=E_\beta\cap\alpha$. Since subtle cardinals are Mahlo, the next observation slightly strengthens {\cite[Theorem 2.10]{ABL}}.

\begin{prop}\label{prop:Subtle}
    Exacting cardinals are subtle in $\HOD$. 
\end{prop}

\begin{proof}
    Assume, towards a contradicting that there is an exacting cardinal $\lambda$ that is not subtle in $\HOD$. Let $\langle C,\vec{E}\rangle$ be the least pair in the canonical well-ordering of $\HOD$ with the property that $C$ is a closed unbounded subset of $\lambda$ and $\vec{E}=\seq{E_\alpha}{\alpha<\lambda}$ is a sequence with $E_\alpha\subseteq\alpha$ for all $\alpha<\lambda$ and $E_\alpha\neq E_\beta\cap\alpha$ for all $\alpha,\beta\in C$ with $\alpha<\beta$. Since the class of all proper initial segments of the canonical well-ordering of $\HOD$ is definable by a $\Sigma_2$-formula without parameters, it follows that the sets $\{C\}$ and $\{\vec{E}\}$ are definable by a $\Sigma_3$-formula with parameter $\lambda$. 

    Pick $\lambda<\zeta\in C^{(3)}$. Our assumptions now yield an elementary submodel $X$ of $V_\zeta$ with $V_\lambda\cup\{\lambda\}\subseteq X$ and an elementary embedding $\map{j}{X}{V_\zeta}$ with $j(\lambda)=\lambda$ and $j\restriction\lambda\neq\id_\lambda$. We then know that $C,\vec{E}\in X$ with $j(C)=C$ and $j(\vec{E})=\vec{E}$. Let $\kappa$ denote the minimal ordinal below $\lambda$ with $j(\kappa)>\kappa$.

    \begin{claim*}
       $\kappa\in C$. 
    \end{claim*}

    \begin{proof}[Proof of the Claim]
      Assume, towards a contradiction, that $\kappa$ is not an element of $C$ and set $\alpha=\min(C\setminus\kappa)>\kappa$. We then know that the set $\{\alpha\}$ is definable by a $\Sigma_3$-formula with parameters in $\kappa\cup\{\lambda\}$ and hence it follows that $j(\alpha)=\alpha$. But, this contradicts the fact that $\map{j\restriction V_\lambda}{V_\lambda}{V_\lambda}$ is an elementary embedding and hence, by Kunen's Theorem, no ordinal in the interval $(\kappa,\lambda)$ is fixed by $j$. 
    \end{proof}

    Elementarity now ensures that $j(\kappa)>\kappa$ is also an element of $C$ and $j(E_\kappa)=E_{j(\kappa)}$ holds. Since our setup  implies that $E_{j(\kappa)}\cap\kappa=j(E_\kappa)\cap\kappa=E_\kappa$, we have derived a contradiction.  
\end{proof}

Next, we use Woodin's notion of $\omega$-strong measurability to prove that exacting cardinals have strong large cardinal properties in $\HOD$. At the core of this analysis lies the following result:

\begin{theorem}\label{theorem:ExactingInHOD}
  Let  $\lambda$ be an exacting cardinal with the property that unboundedly many regular cardinals below $\lambda$ are not $\omega$-strongly measurable in $\HOD$. 
  Then, in $\HOD$, for every $A\in V_{\lambda+1}$ and every closed unbounded subset $C$ of $\lambda$, there is $\eta\in C$ with the property that $\langle V_\eta,\in,A\cap V_\eta\rangle$ is an elementary substructure of $\langle V_\lambda,\in,A\rangle$ 
  and there is a non-trivial elementary embedding of $\langle V_\eta,\in,A\cap V_\eta\rangle$ into itself whose critical sequence consists of elements of $C$. 
\end{theorem}

\begin{proof}
  Assume, towards a contradiction, that the above conclusion fails and let $\langle A,C\rangle$ be the least pair in the canonical well-ordering of $\HOD$ with the property that $A\in \HOD_{\lambda+1}$, $C$ is a closed unbounded subset of $\lambda$ in $\HOD$ and for all $\eta\in C$ with the property that $\langle\HOD_\eta,\in,A\cap V_\eta\rangle$ is an elementary substructure of $\langle \HOD_\lambda,\in,A\rangle$, there is no  non-trivial elementary embedding of  $\langle\HOD_\eta,\in,A\cap V_\eta\rangle$ into itself  in $\HOD$ whose critical sequence consists of elements of $C$.

    By our assumption, there is an ordinal $\zeta>\lambda$ such that $V_\zeta$ is sufficiently elementary in $V$, an elementary submodel $X$ of $V_\zeta$ with $V_\lambda\cup\{\lambda\}\subseteq X$ and an elementary embedding $\map{j}{X}{V_\lambda}$ with $j(\lambda)=\lambda$ and $j\restriction\lambda\neq\id_\lambda$. 
    Then $A,C\in X$ with $j(A)=A$ and $j(C)=C$. 
    Let $\seq{\kappa_n}{n<\omega}$ denote the critical sequence of $j$. 

    \begin{claim*}
        $\kappa_n\in C$ for all $n<\omega$.
    \end{claim*}

    \begin{proof}[Proof of the claim]
      Since there exists an $n<\omega$ with $C\cap\kappa_n\neq\emptyset$, we can use the fact that $j(C)=C$ holds to conclude that $C\cap\kappa_0\neq\emptyset$. Then $\kappa_0\in C$, because otherwise we would have $\sup(C\cap\kappa_n)=\sup(C\cap\kappa_0)<\kappa_0$ for all $n<\omega$ and it would follow that $C$ is bounded in $\lambda$. We can now inductively conclude that $\kappa_n$ is an element of $C$ for all $n<\omega$. 
    \end{proof}

  \begin{claim*}
      If $n<\omega$, then $\langle\HOD_{\kappa_n},\in,A\cap V_{\kappa_n}\rangle$ is an elementary substructure of $\langle\HOD_\lambda,\in,A\rangle$. 
  \end{claim*}

  \begin{proof}[Proof of the claim]
    Note that, since $V_\lambda\cup\{\lambda\}\subseteq X$ and $V_\zeta$ was chosen to be  sufficiently elementary in $V$, we have $\HOD_\eta\in X$ with $j(\HOD_\eta)=\HOD_{j(\eta)}$ for all $\eta\leq\lambda$. This implies that for all $\eta\leq\lambda$, the map $j\restriction \HOD_\eta$ is an elementary embedding of $\langle \HOD_\eta,\in,A\cap V_\eta\rangle$ into $\langle \HOD_{j(\eta)},\in,A\cap V_{j(\eta)}\rangle$. In particular, it follows that $\langle\HOD_{\kappa_0},\in,A\cap V_{\kappa_0}\rangle$ is an elementary substructure of $\langle\HOD_{\kappa_1},\in,A\cap V_{\kappa_1}\rangle$ and, in combination with our first observation,  this inductively implies that $\langle\HOD_{\kappa_n},\in,A\cap V_{\kappa_n}\rangle$ is an elementary substructure of $\langle\HOD_{\kappa_{n+1}},\in,A\cap V_{\kappa_{n+1}}\rangle$ for all $n<\omega$. This conclusion implies the statement of the claim. 
   \end{proof}

    \begin{claim*}
     $j\restriction\eta\in\HOD$ for all $\eta<\lambda$. 
    \end{claim*}

    \begin{proof}[Proof of the claim]
      Fix a cardinal $\eta<\lambda$. By our assumption, there exists a regular cardinal $(2^\eta)^+<\delta<\lambda$ that is not $\omega$-strongly measurable in $\HOD$. Since $(2^\eta)^\HOD<\delta$, it follows that $\HOD$ contains a partition $\vec{S}=\seq{S_\gamma}{\gamma<\eta}$ of $E^\delta_\omega$ into $\eta$-many  sets that are all stationary in $V$. Then $\vec{S}\in V_\lambda\subseteq X$. Set $j(\vec{S})=\seq{T_\gamma}{\gamma<j(\delta)}$. Then $j(\delta)$ is a regular cardinal below $\lambda$ and $j(\vec{S})\in\HOD$ is a partition of $E^{j(\delta)}_\omega$ into $j(\eta)$-many stationary sets.  

      Now, set $\rho=\sup(j[\delta])<j(\delta)$. A result of Solovay in \cite{MR0379200} then shows that $$j[\eta] ~ = ~ \Set{\gamma<j(\eta)}{\textit{$T_\gamma\cap\rho$ is stationary in $\rho$}}.$$ 
      Since $j(\vec{S})$ is an element of $\HOD$, it follows that the set $j[\eta]$ is ordinal definable and this implies that $j\restriction\eta$ is ordinal definable, because this function is the inverse of the transitive collapse of $j[\eta]$.  
      %
    \end{proof}

    \begin{claim*}
     $j\restriction\HOD_\eta\in\HOD$ for all $\eta<\lambda$. 
    \end{claim*}

    \begin{proof}[Proof of the claim]
      Let $\lhd$ denote the ordering of $\HOD_\lambda$ that is obtained by first  ordering sets by their rank and then ordering  sets of the same rank using the canonical well-ordering of $\HOD$. Then $\lhd$ is an element of $\HOD$. Moreover, if $\delta<\lambda$ is  an inaccessible cardinal, then $\HOD_\delta$ is the initial segment of $\lhd$ of order-type $\delta$. For each such cardinal $\delta$, we let $$\map{\pi_\delta}{\langle \HOD_\delta,\lhd\rangle}{\langle\delta,<\rangle}$$ denote the corresponding transitive collapse. Then $\pi_\delta\in\HOD\cap X$ with $j(\pi_\delta)=\pi_{j(\delta)}$ and $$j\restriction \HOD_\delta ~ = ~ \pi_{j(\delta)}^{{-}1} ~ \circ ~ (j\restriction\delta) ~ \circ ~ \pi_\delta$$ for all inaccessible $\delta<\lambda$. In combination, this shows that $j\restriction \HOD_\delta$ is an element of $\HOD$.  
    \end{proof}

  Working in $\HOD$, we define $T$ to be set of all partial elementary embeddings $$\pmap{i}{\langle V_\lambda,\in,A\rangle}{\langle V_\lambda,\in,A\rangle}{part}$$ with the property that there exists a natural number $n$ and a strictly increasing sequence $\seq{\mu_m}{m\leq n+1}$ of elements of $C$ such that $\dom(i)=V_{\mu_n}\cup\{\mu_n\}$, $i\restriction\mu_0=\id_{\mu_0}$,  $i(\mu_m)=\mu_{m+1}$ for all $m\leq n$ and $\langle V_{\mu_m},\in,A\cap V_{\mu_m}\rangle$ is an elementary submodel of $\langle V_\lambda,\in,A\rangle$ for all $m\leq n+1$. By ordering $T$ under extensions, we can view this set as a tree of height at most $\omega$. 

  Since $j\restriction\HOD_\lambda$ is an elementary embedding of $\langle\HOD_\lambda,\in,A\rangle$ into itself, it follows that for all $n<\omega$, the sequence $\seq{\kappa_m}{m\leq n+1}$ witnesses that the function $j\restriction(\HOD_{\kappa_n}\cup\{\kappa_n\})$ is an element of $T$. 
  In particular, there is a cofinal branch through $T$ in $V$ and we can use a well-foundedness argument to conclude that there is such a  branch $b$ in $\HOD$. The definition of $T$ now ensures that there is $\eta\in C$ with the property that $\langle\HOD_\eta,\in,A\cap V_\eta\rangle$ is an elementary substructure of $\langle\HOD_\lambda,\in,A\rangle$ and   $\bigcup b$ is a non-trivial  elementary embedding of $\langle \HOD_\eta,\in,A\cap V_\eta\rangle$ into itself in $\HOD$ whose critical sequence consists of elements of $C$, contradicting our choice of $A$ and $C$. 
\end{proof}

\begin{cor}\label{cor:ExactingStationaryLimitsHOD}
    Exacting cardinals are stationary limits of measurable cardinals in $\HOD$. 
\end{cor}

\begin{proof}
    Let $\lambda$ be an exacting cardinal. Assume, towards a contradiction that the statement of the corollary fails and let $C$ be the least closed unbounded subset of $\lambda$ in the canonical well-ordering of $\HOD$ whose elements are not measurable cardinals in $\HOD$. 

    First, assume that unboundedly many regular cardinals below $\lambda$ are not $\omega$-strongly measurable in $\HOD$. Then Theorem \ref{theorem:ExactingInHOD} yields $\eta\in C$ and a non-trivial elementary embedding of $\langle\HOD_\eta,\in\rangle$ into itself in $\HOD$ whose critical sequence consists of elements of $C$. We then know that $\crit(j)\in C$ is a measurable cardinal in $\HOD$. 

    Now, assume that eventually all regular cardinals below $\lambda$ are $\omega$-strongly measurable in $\HOD$. Pick $\zeta>\lambda$ such that $V_\zeta$ is sufficiently elementary in $V$ and use the exactingness of $\lambda$ to find an elementary submodel $X$ of $V_\zeta$ with $V_\lambda\cup\{\lambda\}\subseteq X$ and an elementary embedding $\map{j}{X}{V_\zeta}$ with $j\restriction\lambda\neq\id_\lambda$ and $j(\lambda)=\lambda$. 
    Let $\seq{\kappa_n}{n<\omega}$ denote the critical sequence of $j$. 
    Our choice of $C$ then ensures that $C$ is an element of $X$ with $j(C)=C$ and we can repeat an argument from the proof of Theorem \ref{theorem:ExactingInHOD} to see that $\kappa_n\in C$ holds for all $n<\omega$. Since the critical sequence of $j$ consists of cardinals that are regular in $V$, our assumption yields an $n<\omega$ with the property that $\kappa_n$ is $\omega$-strongly measurable in $\HOD$. In particular, it follows that $\kappa_n$ is an element of $C$ that is a measurable cardinal in $\HOD$. 
\end{proof}

We now observe that the assumption of Theorem \ref{theorem:ExactingInHOD} are consequences of the validity of the $\HOD$ Conjecture:

\begin{cor}\label{corollary:ExactingInHOD}
    Let $\lambda$ be an exacting cardinal with the property that the $\HOD$ Hypothesis holds in $V_\lambda$, let $A\in\HOD_{\lambda+1}$ and let $C$ be a closed unbounded subset of $\lambda$ in $\HOD$. 
      Then, in $\HOD$, for every $A\in V_{\lambda+1}$ and every closed unbounded subset $C$ of $\lambda$, there is $\eta\in C$ with the property that $\langle V_\eta,\in,A\cap V_\eta\rangle$ is an elementary substructure of $\langle V_\lambda,\in,A\rangle$ 
  and there is a non-trivial elementary embedding of $\langle V_\eta,\in,A\cap V_\eta\rangle$ into itself whose critical sequence consists of elements of $C$. 
\end{cor}

\begin{proof}
    If $\kappa<\lambda$ is a cardinal that is not $\omega$-strongly measurable in $\HOD$ in $V_\lambda$, then the fact that $\HOD^{V_\lambda}\subseteq\HOD$ holds ensures that $\kappa$ is not $\omega$-strongly measurable in $\HOD$ in $V$. In particular, if the $\HOD$ Hypothesis holds in $V_\lambda$, then there are  unboundedly many regular cardinals below $\lambda$ that are not $\omega$-strongly measurable in $\HOD$. The desired conclusion now follows from an application of Theorem \ref{theorem:ExactingInHOD}. 
\end{proof}

Corollary \ref{corollary:ExactingInHOD} directly implies the related result stated in  \S\ref{sec:Intro}:

\begin{proof}[Proof of Theorem \ref{theorem:ExactingInHODIntro}]
  Assume that the Weak $\HOD$ Conjecture is true and let $\lambda$ be an exacting cardinal. Then $\lambda$ is an inaccessible cardinal in $\HOD$. Moreover, since $V_\lambda$ is a model of $\ZFC$ with a huge cardinal above an extendible cardinal. Our assumption then implies that the $\HOD$ Hypothesis holds in $V_\lambda$ and this directly yields \eqref{item:ExactingInHODIntro2} of the theorem. 

    For \eqref{item:ExactingInHODIntro1}, let $\Gamma\in \HOD_{\lambda+1}\setminus V_\lambda$ be a set of graphs and let $C\in\HOD$ be the closed unbounded subset of $\lambda$ consisting of the limits of ranks of graphs in $\Gamma$. An application of Corollary \ref{corollary:ExactingInHOD} now yields a non-trivial elementary embedding $j$ of $\langle \HOD_\eta,\in,\Gamma\cap V_\eta\rangle$ into itself in $\HOD$ for some $\eta\in C$. By the definition of $\Gamma$, there exists a graph $\mathcal{G}$ in $\Gamma\cap V_\eta$ whose rank is greater than the critical point of $j$. Since $\eta$ is the least non-trivial fixed point of $j$, it follows that $\mathcal{G}\neq j(\mathcal{G})\in \Gamma$ and $j$ induced an embedding of $\mathcal{G}$ into $j(\mathcal{G})$. Since $\lambda$ is inaccessible in $\HOD$, these computations show that $\lambda$ is a Vop\v{e}nka cardinal in $\HOD$. 
%
\end{proof}

The concepts used in the above proof can also be applied to derive similar conclusions for rank-into-rank embeddings using inverse limits: 


\begin{theorem}\label{theorem:RankIntoRankInHOD}
    Let $\map{j}{V_\lambda}{V_\lambda}$ be a $\Sigma^1_4$-correct I3-embedding and let $A\in\HOD_{\lambda+1}$. If the $\HOD$ Hypothesis holds in $V_{\bar{\lambda}}$ for all $\bar{\lambda}<\lambda$ with the property that there exists an I3-embedding from $V_{\bar{\lambda}}$ to itself, then $\lambda$ is a limit of ordinals $\alpha$ such that, in $\HOD$, there is a non-trivial elementary embedding of the structure $\langle V_\alpha,\in,A\cap V_\alpha\rangle$ into itself. 
\end{theorem}

\begin{proof}
  Fix $\xi<\lambda$. Using {\cite[Theorem 2.4]{La97}}, we can find a $\Sigma^1_4$-correct I3-embedding $\map{j^\prime}{V_\lambda}{V_\lambda}$ with $\crit(j^\prime)>\xi$. Let $\vec{\kappa}=\seq{\kappa_n}{n<\omega}$ denote the critical sequence of $j^\prime$. We then know that $j^\prime$ can be extended to an I2-embedding that is given by a $(\kappa_0,\lambda)$-extender $E$. Let $$\langle\seq{M_\alpha}{\alpha\in\ord},\seq{\map{j^\prime_{\alpha,\beta}}{M_\alpha}{M_\beta}}{\alpha\leq\beta\in\ord}\rangle$$ denote the iteration of $V$ and $E$ and set $A^\prime=j^\prime_{0,\omega}(A\cap V_{\kappa_0})$. Then $A\cap\kappa_0=A^\prime\cap\kappa_0$ and Proposition \ref{proposition:IterationFixedPoints}\eqref{prop:InvEmb2} ensures that $j^\prime_+(A^\prime)=j^\prime_{0,1}(A^\prime)=A^\prime$. 
  In addition, let $\lhd$ denote the ordering of $\HOD_{\kappa_0}$ that is obtained by first  ordering sets by their rank and then ordering  sets of the same rank  using the canonical well-ordering of $\HOD$. Set ${\blacktriangleleft}=j_{0,\omega}(\lhd)$. Then ${\blacktriangleleft}\cap V_{\kappa_0}=\lhd$ and $j^\prime_+({\blacktriangleleft})=j^\prime_{0,1}({\blacktriangleleft})={\blacktriangleleft}$. 
  
  Iterated applications of {\cite[Theorem 2.4]{La97}} now allow us to find a sequence $\seq{\map{i_m}{V_\lambda}{V_\lambda}}{m<\omega}$ of $\Sigma^1_2$-correct I3-embeddings with the property that the sequence $\seq{\crit(i_m)}{m<\omega}$ is strictly increasing in the interval $(\xi,\kappa_0)$ and for all $m<\omega$, the critical sequence of $i_m$ is equal to $\langle\crit(i_n)\rangle^\smallfrown\vec{\kappa}$,  $(i_m)_+(A^\prime)=A^\prime$, $(i_m)_+({\blacktriangleleft})={\blacktriangleleft}$, $i_m(\crit(i_m))=\kappa_0$ and $i_m(\kappa_n)=\kappa_{n+1}$ for all $n<\omega$. Define  $$\bar{\lambda} ~ = ~ \sup_{n<\omega}\crit(i_n) ~ \in ~ (\xi,\kappa_0).$$ 
  
  An inverse limit construction (see {\cite[Section 3]{La97}}) now yields an elementary  embedding $\map{i}{V_{\bar{\lambda}}}{V_\lambda}$ with critical sequence $\seq{\crit(i_n)}{n<\omega}$ and $i(\crit(i_n))=\kappa_n$ for all $n<\omega$. We can then extend $i$ to a map $\map{i_+}{V_{\bar{\lambda}+1}}{V_{\lambda+1}}$ by setting $$i_+(B) ~ = ~ \bigcup\Set{i(B\cap V_{\crit(i_n)})}{n<\omega}$$ for all $B\in V_{\bar{\lambda}+1}$. Our setup then ensures that $$i_+(A\cap V_{\bar{\lambda}}) ~ = ~ i_+(A^\prime\cap V_{\bar{\lambda}}) ~ = ~ A^\prime$$ and $$i_+({\lhd}\cap V_{\bar{\lambda}}) ~ = ~ i_+({\blacktriangleleft}\cap V_{\bar{\lambda}}) ~ = ~ {\blacktriangleleft}.$$ 
  Using {\cite[Theorem 3.3]{La97}}, we can now find an I3-embedding $\map{k}{V_{\bar{\lambda}}}{V_{\bar{\lambda}}}$ with critical sequence $\seq{\crit(i_n)}{n<\omega}$ satisfying $$k_+(A\cap V_{\bar{\lambda}}) ~ = ~ A\cap V_{\bar{\lambda}}$$ and $$k_+({\lhd}\cap V_{\bar{\lambda}}) ~ = ~ {\lhd}\cap V_{\bar{\lambda}}.$$ 
  
  Note that, since $\HOD_{\bar{\lambda}}$ is the initial segment of $\lhd$ of order-type $\bar{\lambda}$, we also know that $k_+(\HOD_{\bar{\lambda}}) = \HOD_{\bar{\lambda}}$. In particular, it follows that for all $\alpha<\bar{\lambda}$, the map  $k\restriction \HOD_\alpha$ is a non-trivial elementary embedding of the structure $\langle \HOD_\alpha,\in,A\cap V_\alpha\rangle$ into the structure $\langle \HOD_{k(\alpha)},\in,A\cap V_{k(\alpha)}\rangle$. Given $n<\omega$, we then know that $\langle \HOD_{\crit(i_n)},\in,A\cap V_{\crit(i_n)}\rangle$ is an elementary substructure of $\langle \HOD_{\bar{\lambda}},\in,A\cap V_\lambda\rangle$. 

  Since our assumptions ensure that the $\HOD$ Hypothesis holds in $V_{\bar{\lambda}}$, we can repeat the corresponding argument from the proof of Theorem \ref{theorem:ExactingInHOD} to obtain the following statement: 

 \begin{claim*}
     $k\restriction\alpha\in\HOD$ for all $\alpha<\bar{\lambda}$. \qed 
 \end{claim*}  

 If $\delta<\bar{\lambda}$ is an inaccessible cardinal, then $\HOD_\delta$ is the initial segment of $\lhd$ or order-type $\delta$. Given such a cardinal $\delta$, we let $$\map{\pi_\delta}{\langle \HOD_\delta,\lhd\rangle}{\langle\delta,<\rangle}$$ denote the corresponding transitive collapse. The above computations now that $k(\pi_\delta)=\pi_{k(\delta)}$ holds. This allows us to repeat another argument from the proof of Theorem \ref{theorem:ExactingInHOD} and derive the following statement: 

  \begin{claim*}
     $k\restriction \HOD_\alpha\in\HOD$ for all $\alpha<\bar{\lambda}$. \qed 
 \end{claim*}  

 Now, in $\HOD$, we define $T$ to be set of all partial elementary embeddings $$\pmap{\ell}{\langle V_{\bar{\lambda}},\in,A\cap V_{\bar{\lambda}}\rangle}{\langle V_{\bar{\lambda}},\in,A\cap V_{\bar{\lambda}}\rangle}{part}$$ with the property that there exists a natural number $n$ and a strictly increasing sequence $\seq{\mu_m}{m\leq n+2}$ of ordinals below $\bar{\lambda}$ such that $\dom(\ell)=V_{\mu_{n+1}}$, $\ran{\ell}=V_{\mu_{n+2}}$, $\ell\restriction\mu_0=\id_{\mu_0}$,  $\ell(\mu_m)=\mu_{m+1}$ for all $m\leq n$ and the structure $\langle V_{\mu_m},\in,A\cap V_{\mu_m}\rangle$ is an elementary substructure of $\langle V_{\bar{\lambda}},\in,A\cap V_{\bar{\lambda}}\rangle$ for all $m\leq n+2$. 
 Ordering the elements of  $T$ under extensions then turns this set into a tree of height at most $\omega$. Moreover, for all $n<\omega$, the function $k\restriction \HOD_{\crit(n+1)}$ is an element of $T$. This shows that $T$ has a cofinal branch in $V$ and hence this tree has a cofinal branch $b$ in $\HOD$. In this situation, there exists an ordinal $\xi<\alpha\leq\bar{\lambda}$ with the property that $\bigcup b$ is a non-trivial elementary embedding of the structure $\langle \HOD_\alpha,\in,A\cap V_\alpha\rangle$ into itself in $\HOD$.  
\end{proof}


\subsection{Successors of exacting cardinals in HOD}\label{sec: successors of exactings}
We now investigate the possible large cardinal properties of the successors of exacting cardinals in $\HOD$. Since every model of $\ZFC$ with an I2-embedding has a class forcing extension that is a model of $\ZFC+\gch$ and contains an I2-embedding (see {\cite[Theorem 2]{MR2374763}}), Lemma~\ref{lemma:I2andVHOD} shows that Theorem \ref{theorem: exacting + extendible} is a consequence of the following result:

\begin{theorem}\label{theorem:ConstructionSuccOfExactingExtendibleInZF}
  Assume that $\map{j}{V}{M}$ is an I2-embedding with least non-trivial fixed point $\lambda$ and the $\mathrm{CCA}$ holds in  $V_\lambda$. Then there is a transitive set $N$ such that $V_\lambda\in N$ and the following statements hold: 
 \begin{enumerate}
   \item  $N$ is a model of $\mathrm{ZF}+\dc_{{<}\lambda}$. 
   
     \item  $\lambda$ is an exacting cardinal in $N$. 
     \item $\lambda^+$ is an extendible cardinal in $\HOD_x^N$ for all  $x\in \mathcal{P}(\lambda)^N$.
 \end{enumerate}
\end{theorem}

\begin{proof}
 Let $\seq{\kappa_n}{n<\omega}$ denote the critical sequence of $j$. Then $\kappa=\kappa_0$ and $\lambda=\sup_{n<\omega}\kappa_n$. 
 Without loss of generality, we may also assume that $j$ is given by a $(\kappa,\lambda)$-extender $E$ and, as outlined above, we can consider the iteration $$\langle \seq{M_\alpha}{\alpha\in\ord},\seq{\map{j_{\alpha,\beta}}{M_\alpha}{M_\beta}}{\alpha\leq\beta\in\ord}\rangle$$ of $V$ and $E$. Let $\kappa <\mu <j(\kappa)$ be an extendible cardinal in $V_\lambda$ and let $\calU$ be the normal, ${<}\kappa$-complete, fine ultrafilter on $\mathcal{P}_\kappa(\mu)$ induced by $j$. In addition, let $\mathbb{P}_{\mathcal{U}}$ be the corresponding {Supercompact Prikry forcing} (see \S\ref{sec: Supercompact Prikry}).

\smallskip

Define $\vec{x}=\seq{x_n}{n<\omega}$ to be the unique sequence with $$x_n ~  = ~ j_{n,\omega}[j_{0,n}(\mu)]$$ for all $n<\omega$.\footnote{Note that $j_{n,n+1}[j_{0,n}(\mu)]$ equals $[\id]_{j_{0,n}(\mathcal{U})}$.}

\begin{claim*}
    The sequence $\vec{x}$ is $\prec$-increasing in $\mathcal{P}_\lambda(j_{0,\omega}(\mu))$ and for each $A\in j_{0,\omega}(\calU)$, there is an $n_*<\omega$ with the property that $x_n\in A$ for all $n\geq n_*$.\qed
\end{claim*}

By the Mathias criterion of genericity for Supercompact Prikry forcing (see \cite[\S1]{Gitik-handbook}), the sequence $\vec{x}$ induces a $j_{0,\omega}(\PPP_{\calU})$-generic filter $$G(\vec{x}) ~ = ~ \Set{\langle s, A\rangle\in j_{0,\omega}(\PPP_{\calU})}{s\sqsubseteq \vec{x}, ~ \vec{x}\setminus s\s A}$$ over $M_\omega$. 
Clearly, we have $M_\omega[\vec{x}]=M_\omega[G(\vec{x})]$ and  $|j_{0,\omega}(\PPP_{\calU})|^{M_\omega}<j_{0,\omega}(\lambda)$. 
In particular, it follows that $(M_{\omega})_{j_{0,\omega}(\lambda)}[\vec{x}]=M_\omega[\vec{x}]_{j_{0,\omega}(\lambda)}$ is a model of $\ZFC$. 
Define $$\bar{M} ~ = ~ (M_\omega)_{j_{0,\omega}(\lambda)}.$$ Elementarity then ensures that $\bar{M}$ is a model of $\ZFC+``V=\gHOD"$ and $j_{0,\omega}(\mu)$ is extendible in $\bar{M}$.\footnote{In \cite{ABLG}, the authors  looked at the generic extension of $\bar{M}$ via the Prikry sequence induced by $j$. The natural analogue here would be to look at the generic extension $\bar{M}[\vec{x}]$. The caveat, though, is that the $\bar{M}$-cardinal $j_{0,\omega}(\mu)$ is collapsed to $\lambda$ when passing from $\bar{M}$  to $\bar{M}[\vec{x}]$. So instead, we will look at the $\ZF$-submodel of $\bar{M}[\vec{x}]$ given by all the objects that are  definable (in $\bar{M}[\vec{x}]$) from parameters in $\bar{M}$ and from subsets of $\lambda$ in the various sub-generic extensions of $\bar{M}[\vec{x}]$.} Moreover, we have $V_\lambda\in\bar{M}$.

Define  $\mathcal{R}$ to be the collection of all $\bar{M}$-regular cardinals in the interval $[\lambda,j_{0,\omega}(\mu))$  and for each  $\alpha\in \mathcal{R}$, define $$\map{\pi_\alpha}{\mathcal{P}_\lambda(j_{0,\omega}(\mu))}{\mathcal{P}_\lambda(\alpha)}; ~ x\mapsto x\cap \alpha$$ to be the map  that induces a Rudin--Keisler projection of $j_{0,\omega}(\calU)$ that we denote by  $j_{0,\omega}(\mathcal{U})_\alpha$. 
By the Mathias criterion of genericity, the sequence  $$\vec{x}_\alpha  ~ = ~ \seq{x_n\cap\alpha}{n<\omega}$$ is $\PPP_{j_{0,\omega}(\calU)_\alpha}$-generic over $\bar{M}$. 

\smallskip

Following arguments of Kafkoulis in \cite{Kafkoulis}, our intended model is $$N ~ = ~ \bar{M}\big(\bigcup\Set{\mathcal{P}^{\bar{M}[\vec{x}_\alpha]}(\lambda)}{\alpha\in \mathcal{R}}\big).$$

The following statements are all proven by Kafkoulis in \cite{Kafkoulis}:

\begin{claim*}
    \begin{enumerate}
        \item\label{item:Kafkoulis1} $N$ is a model of $\ZF+\dc_{{<}\lambda}$. 
        
        \item\label{item:Kafkoulis2} Every set of ordinals in $N$ belongs to $\bar{M}[\vec{x}_\alpha]$ for some $\alpha\in \mathcal{R}$. 
        
        \item\label{item:Kafkoulis3} $\lambda$ is a strong limit cardinal of countable cofinality in $N$ and $j_{0,\omega}(\mu)=(\lambda^+)^N$. 
        
        \item\label{item:Kafkoulis4} The following statements are equivalent for every $\theta <\lambda$: 
          \begin{enumerate}
              \item $\theta$ is an $N$-cardinal. 

              \item $\theta$ is an $\bar{M}[\vec{x}]$-cardinal. 

              \item $\theta$ is an $\bar{M}$-cardinal.
          \end{enumerate} 
    \end{enumerate}
\end{claim*}

\begin{proof}[Proof of the claim]
    \eqref{item:Kafkoulis1}  is {\cite[Corollary 4.1.13]{Kafkoulis}},  
    \eqref{item:Kafkoulis2}  is {\cite[Corollary 2.1.12]{Kafkoulis}},   \eqref{item:Kafkoulis3} is \cite[Corollary 2.1.13]{Kafkoulis} and \eqref{item:Kafkoulis4} follows from the fact that forcing with $j_{0,\omega}(\PPP_{\calU})$ does not add bounded subsets to $\lambda$. 
\end{proof}

\begin{claim*}
    $j\restriction N$ is an elementary embedding of $N$ into itself. 
\end{claim*}

\begin{proof}[Proof of claim]
  Since $N$ is a set-sized model and $\map{j}{V}{M}$ is an elementary embedding, it is clear that $\map{j\restriction N}{N}{j(N)}$ is also an elementary embedding. Thus, it suffices to argue that $j(N)=N$ holds.  To show this, we first observe the  following: 
  \begin{itemize}
    \item $j(\bar{M})=j((M_\omega)_{j_{0,\omega}(\lambda)})=(M_\omega)_{j_{0,\omega}(\lambda)}$. 
    
    \item $j(\mathcal{R})=\mathcal{R}$. 
    
    \item $j(\vec{x})_\alpha=\seq{j(x_n)\cap\alpha}{n<\omega}=\seq{x_{n+1}\cap\alpha}{n<\omega}$.\footnote{In particular, we have $\bar{M}[\vec{x}_\alpha]=\bar{M}[j(\vec{x})_\alpha]$, because $x_0=x_{1}\cap \sup(x_0)$.}
  \end{itemize}

 The set $N$ is the smallest $\ZF$-model containing $\bar{M}$ that has the set $$\bigcup\Set{\mathcal{P}^{\bar{M}[\vec{x}_\alpha]}(\lambda)}{\alpha\in \mathcal{R}}$$ as an element. By elementarity, in $M$, the set $j(N)$ is  the smallest $\ZF$-model containing $j\big(\bar{M})=\bar{M}$ that has $$j(\bigcup\Set{\mathcal{P}^{\bar{M}[\vec{x}_\alpha]}(\lambda)}{\alpha\in \mathcal{R}}\big) ~ = ~ \bigcup\Set{\mathcal{P}^{\bar{M}[\vec{x}_\alpha]}(\lambda)}{\alpha\in \mathcal{R}}$$ as an element. In particular, we have $j(N)=N$, as needed.
\end{proof}

We now know that $N\subseteq M$ is a model of $\ZF$, $V_\lambda\in N$, $\cof{\lambda}^N=\omega$ and $j\restriction N$ is an elementary embedding of $N$ into $N$. Moreover, since $V_\lambda\in N$ is a model of $\mathrm{CCA}$ and therefore $"V=\HOD"$ holds in $V_\lambda$,  there is a well-ordering of $V_\lambda$ in $N$. 
Invoking Corollary~\ref{corollary:ExactingInnerModel}, we can  conclude that  $\lambda$ is an  exacting cardinal in $N$.

\smallskip

Finally, we analyze $\HOD^N_x$ for an arbitrary set $x\in \mathcal{P}^N(\lambda)$:

\begin{claim*}
    For each $x\in \mathcal{P}^N(\lambda)$ there is $\alpha\in \mathcal{R}$ such that $$\bar{M}\s \HOD^N_x\s \bar{M}[\vec{x}_\alpha].$$
\end{claim*}

\begin{proof}[Proof of claim]
   To prove the first inclusion, it suffices to show that every set of ordinals in $\bar{M}$ belongs to $\HOD^N_x$. Since every set of ordinals in $\bar{M}$ was coded into the continuum-function-pattern cofinally often  below $j_{0,\omega}(\lambda)$ and $|j_{0,\omega}(\mathbb{P}_{\mathcal{U}})|<j_{0,\omega}(\lambda)$, this coding is absolute between $\bar{M}$ and $N$. In particular, we have  $$\bar{M} ~ \s  ~ \HOD^N ~ \s ~ \HOD^N_x.$$ 

   For the second  inclusion, since $\bar{M}\s \HOD^N_x\s \bar{M}[\vec{x}]$, the intermediate model theorem hands us a generic filter $H\in\bar{M}[\vec{x}]$ for a subforcing of $j_{0,\omega} (\mathbb{P}_{\mathcal{U}})$ such that $\bar{M}[H]=\HOD^N_x$. Since the latter is a model of choice, there is a set of ordinals $A\in N$ such that $\bar{M}[H]=\bar{M}[A]$. Invoking \eqref{item:Kafkoulis2} from our second claim,  we find $\alpha\in \mathcal{R}$ such that $A\in \bar{M}[\vec x_\alpha]$, hence $\HOD^N_x\s \bar{M}[\vec{x}_\alpha]$, as needed. 
\end{proof}

By the intermediate model theorem $\HOD^N_x$ is a generic extension by a subforcing of $\mathrm{RO}(j_{0,\omega}(\mathbb{P})_{\pi_\alpha(j_{0,\omega}(\mathcal{U}))})$, hence by a forcing of size less than $j_{0,\omega}(\mu)$. As a result, the latter is an extendible cardinal in $\HOD^N_x$. 
\end{proof}

The following observation  sharpens \cite[Corollary~5.6]{ABLG}:

\begin{theorem}\label{thm:generalizingABLG}
   Assume the  hypothesis of Theorem \ref{theorem:ConstructionSuccOfExactingExtendibleInZF}. Then there is a transitive, set-sized model $\bar{M}$ of $\ZFC$ with $V_\lambda\in\bar{M}$ and an ordinal $\lambda<\bar{\mu}\in\bar{M}$ with the property that for every $\lambda<\alpha<\bar{\mu}$, there exists a transitive set $M(\alpha)$ such that $M(\alpha)$ is a forcing extension of $\bar{M}$, $\lambda$ is an exacting cardinal in $M(\alpha)$, $\bar{\mu}$ is an extendible cardinal in $M(\alpha)$ and $\alpha<(\lambda^+)^{M(\alpha)}$.  
\end{theorem}

\begin{proof}
    We use the notations of the proof of Theorem~\ref{theorem: exacting + extendible} and again define $\bar{M}=(M_\omega)_{j_{0,\omega}(\lambda)}$. Set $\bar{\mu}=j_{0,\omega}(\mu)$. Given $\lambda<\alpha<\bar{\mu}$, pick $\alpha\leq\beta\in\mathcal{R}$ and define $M(\alpha)=\bar{M}[\vec{x}_\beta]$. 
    The above arguments then show that $\bar{\mu}$ is an extendible cardinal in $M(\alpha)$. Moreover, Corollary~\ref{corollary:ExactingInnerModel} ensures that $\lambda$ is exacting in $M(\alpha)$. 
\end{proof}

The formulation of Theorem \ref{theorem: exacting + extendible} was motivated by {\cite[Theorem 3.14]{ABL}} showing that the successors of ultraexacting cardinals are measurable cardinals in $\HOD$. 
The following result shows that $\ZFC$ does not prove that these successors are extendible in $\HOD$.

\begin{prop}
 If $\mathrm{ZFC}$ is consistent with the existence of an ultraexacting cardinal, then the resulting theory does not prove that a proper class of ordinals are measurable in $\HOD$. 
\end{prop}

\begin{proof}
   By \cite[Theorem 4.5]{ABLG}, the consistency of $\mathrm{ZFC}$ with an ultraexacting cardinal implies the consistency of $\mathrm{ZFC}$ with an I0-embedding. Work in a model of $\mathrm{ZFC}$ with an I0-embedding $\map{j}{V_{\lambda+1}}{V_{\lambda+1}}$ and let $G$ be $\Add{\lambda^+}{1}$-generic over $V$. Then $L(V_{\lambda+1})[G]$ is a model of $\mathrm{ZFC}$ and \cite[Corollary 3.31]{ABL} shows that $\lambda$ is an ultraexacting cardinal in this model. Classical results of Vop\v{e}nka (see \cite[Theorem~6]{MR3821636}) now show that $L(V_{\lambda+1})[G]$ is a forcing extension of $\HOD^{L(V_{\lambda+1})[G]}$ using a set-sized partial order and, since $L(V_{\lambda+1})[G]$ does not contain a proper class of measurable cardinals, it follows that $\HOD^{L(V_{\lambda+1})[G]}$ does not contain a proper class of measurable cardinals. 
\end{proof}


\section{Open questions}\label{sec:openproblems}

The results of this paper leave all four far-reaching questions posed in  \S\ref{sec:Intro} unresolved. Below, we present open problems whose solutions should provide substantial progress towards answering these questions. 

\smallskip

In Theorem~\ref{theorem: WeakHODConjecture and LevySolovay}, we showed that the Weak HOD Conjecture answers Question \ref{Q03} in the negative, {i.e.,} it implies that no partial order $\mathbb{P}\in H(\lambda)$ can make a cardinal $\lambda$ of uncountable cofinality exacting. Motivated by Theorems \ref{thm:SimplifiedConsStatLimitUltra}  and \ref{theorem:LevySolovayBeyondHOD}, we ask whether  assumptions --whose consistency is unclear-- can be used to obtain settings in which new exacting cardinals are produced by small forcings:

\begin{question}\label{que:isWHODHnecessary}
  Assuming the consistency of $\ZF$ with some large cardinal axioms that is not known to be inconsistent, is it possible to show that $\ZFC$ is consistent with the existence of a cardinal  $\lambda$ of uncountable cofinality  with the property that $\one\Vdash_\PPP\textit{$``\check{\lambda}$ is an exacting cardinal }"$ holds for some  partial order $\PPP\in H(\lambda)$? 
\end{question}

Next, we consider ramifications of Question \ref{Q04}. Theorem \ref{theorem:ExactingInHODIntro} and Proposition \ref{prop:Subtle} lead us to the following open problem:

\begin{question}
    If the $\HOD$ Hypothesis is true, is every exacting cardinal weakly compact in $\HOD$?
\end{question}

In addition, we can again ask if some version of the $\HOD$ Conjecture is necessary for the conclusion of Theorem \ref{theorem:ExactingInHOD}:

\begin{question}
  Does the existence of an exacting cardinal imply the existence of an I3-embedding in $\HOD$? 
\end{question}

On a similar vein, we may ask about variations of Theorem \ref{theorem:RankIntoRankInHOD} that do not rely on the $\HOD$ Conjecture:

\begin{question}
 Does the existence of an I1-embedding imply the existence of an I3-embedding in $\HOD$? If so, does this conclusion also follow from the existence of an I2- or I3-embedding? 
\end{question}

The reader might be surprised that the model $N$ obtained in Theorem~\ref{theorem: exacting + extendible}  is not a  model of $\ZFC$. This raises the following question:

\begin{question}
    Is there a version of  Theorem~\ref{theorem: exacting + extendible}  where $N$ is a model of $\ZFC$?
\end{question}

A natural strategy for circumventing this issue would be to replace Supercompact Prikry forcing with the Merimovich poset from \cite{MerimovichSupercompact}, relative to the $(\kappa,\mu)$-extender derived from $\map{j}{V}{M}$. Indeed, using the ideas of Merimovich from \cite[Theorem~5.3]{MerimovichPrikryOn}, one can construct, by iterating the original I2-embedding, an $M_\omega$-generic filter $G$. The difficulty, however, is that we need to show that $\lambda$ is exacting in $M_\omega[G]$. For this, we must verify that $j\restriction M_\omega[G]$ is an elementary embedding from $M_\omega[G]$ into itself. Although $j$ fixes $M_\omega$, it is not clear that it also fixes $G$, or even that $M_\omega[G]=M_\omega[j(G)]$.

One possible approach would be to establish that the Bukovský--Dehornoy \cite{Bukovsky,Dehornoy} phenomenon holds in this context; namely, that $M_\omega[G]=\bigcap_{n<\omega}M_n$, where $\map{j_n}{V}{M_n}$ denotes the $n$-th iterate of the original embedding. While the $\s$-inclusion is clear, the converse is not. Establishing it would require $M_\omega[G]$ to be closed under $\omega$-sequences in $V$ (see, e.g., \cite[Theorem~18]{HayutBukovski}). However, this closure fails in our context, since $\map{j}{V}{M}$ is an I2-embedding and hence $j[\lambda]\notin M$ by the Kunen Inconsistency.


\section*{Acknowledgements}
The first author presented the results of this paper at the Simons Semester \emph{``Gödel's Program''} held at IMPAN in June 2026. The authors are very grateful to the organizers for their kind invitations as well as to the participants for their comments and constructive feedback. Lücke gratefully acknowledges
support from the \emph{Deutsche Forschungsgemeinschaft} (Project numbers 522490605 and 567184449). 
Poveda was supported by project  PID2023-147428NB-I00 from the Spanish Government, by the \emph{Alexander von Humboldt Foundation} and by Fundación BBVA\footnote{The BBVA Foundation assumes no responsibility for the opinions, comments, or content included in the project and/or in any results derived from it, all of which are the sole and exclusive responsibility of their authors.} through \emph{Beca Leonardo de Investigación Científica y Creación Cultural 2026}.

\bibliographystyle{alpha}
\bibliography{references}
\end{document}